\documentclass[oneside,reqno]{amsart}
\usepackage[T1]{fontenc}
\usepackage[latin9]{inputenc}
\usepackage{color}
\usepackage{mathrsfs}
\usepackage{amsthm}
\usepackage{amstext}
\usepackage{amssymb}
\usepackage{esint}
\usepackage[unicode=true,pdfusetitle,
 bookmarks=true,bookmarksnumbered=false,bookmarksopen=false,
 breaklinks=false,pdfborder={0 0 1},backref=false,colorlinks=false]
 {hyperref}
\hypersetup{
 colorlinks=true,citecolor=blue,linkcolor=blue,linktocpage=true}

\makeatletter
\numberwithin{equation}{section}
\numberwithin{figure}{section}
\theoremstyle{plain}
\newtheorem{thm}{\protect\theoremname}[section]
  \theoremstyle{plain}
  \newtheorem{prop}[thm]{\protect\propositionname}
  \theoremstyle{plain}
  \newtheorem{lem}[thm]{\protect\lemmaname}
  \theoremstyle{remark}
  \newtheorem{rem}[thm]{\protect\remarkname}
  \theoremstyle{plain}
  \newtheorem{cor}[thm]{\protect\corollaryname}
  \theoremstyle{definition}
  \newtheorem{example}[thm]{\protect\examplename}
  \theoremstyle{definition}
  \newtheorem{problem}[thm]{\protect\problemname}
  \theoremstyle{plain}
  \newtheorem*{thm*}{\protect\theoremname}

\renewcommand{\section}{%
\@startsection{section}{1}%
  \z@{.7\linespacing\@plus\linespacing}{.5\linespacing}%
  {\normalfont\scshape\centering\bfseries}}

\renewcommand{\subsection}{%
\@startsection{subsection}{2}%
  \z@{.5\linespacing\@plus.7\linespacing}{.5\linespacing}%
  {\normalfont\bfseries}}

\makeatother

  \providecommand{\corollaryname}{Corollary}
  \providecommand{\examplename}{Example}
  \providecommand{\lemmaname}{Lemma}
  \providecommand{\problemname}{Problem}
  \providecommand{\propositionname}{Proposition}
  \providecommand{\remarkname}{Remark}
  \providecommand{\theoremname}{Theorem}
\providecommand{\theoremname}{Theorem}

\begin{document}

\title{Momentum Operators in the Unit Square}

\author{Steen Pedersen and Feng Tian}

\address{(Steen Pedersen) Department of Mathematics, Wright State University,
Dayton, OH 45435, U.S.A. }

\email{steen@math.wright.edu }

\urladdr{http://www.wright.edu/\textasciitilde{}steen.pedersen/}

\address{(Feng Tian) Department of Mathematics, Wright State University, Dayton,
OH 45435, U.S.A. }

\email{feng.tian@wright.edu }

\urladdr{http://www.wright.edu/\textasciitilde{}feng.tian/}
\begin{abstract}
We investigate the skew-adjoint extensions of a partial derivative
operator acting in the direction of one of the sides a unit square.
We investigate the unitary equivalence of such extensions and the
spectra of such extensions. It follows from our results, that such
extensions need not have discete spectrum. We apply our techniques
to the problem of finding commuting skew-adjoint extensions of the
partial derivative operators acting in the directions of the sides
of the unit square. 

While our results are most easily stated for the unit square, they
are established for a larger class of domains, including certain fractal
domains. 
\end{abstract}


\maketitle
\tableofcontents{}

\section{Introduction\label{sec:Introduction}}

Consider $P_{\min}:=\tfrac{1}{i2\pi}\tfrac{d}{dx}$ acting in $C_{c}^{\infty}\left(\left[0,1\right]\right).$
This operator is symmetric and its selfadjoint extensions are in one-to-one
correspondence with the complex numbers $e\left(\theta\right):=e^{i2\pi\theta},$
$0\leq\theta<1.$ The selfadjoint extension $P_{\theta}$ corresponding
to $e\left(\theta\right)$ has domain
\[
\mathrm{dom}\left(P_{\theta}\right):=\left\{ f\in L^{2}\left(\left[0,1\right]\right)\mid f'\in L^{2}\left(\left[0,1\right]\right),f(1)=e(\theta)f(0)\right\} 
\]
and $P_{\theta}f=\tfrac{1}{i2\pi}f',$ for $f$ in $\mathrm{dom}\left(P_{\theta}\right),$
the derivative is in the distribution sense. The spectrum of $P_{\theta}$
is the set $\theta+\mathbb{Z}:=\left\{ \theta+m\mid m\in\mathbb{Z}\right\} .$
See, for example, \cite{ReSi75}. In particular, $P_{\theta}$ is
unitary equivalent to $P_{\theta}+1$ and $P_{\theta}$ is not unitary
equivalent to $P_{\theta'}$ unless $\theta=\theta'.$ In this paper
we extend this analysis to $\tfrac{1}{i2\pi}\tfrac{\partial}{\partial x}$
acting in the unit square $\left[0,1\right]^{2}.$ We apply our techniques
to investigate the characterization of commuting selfadjoint realizations
(extensions) of $\tfrac{1}{i2\pi}\tfrac{\partial}{\partial x}$ and
$\tfrac{1}{i2\pi}\tfrac{\partial}{\partial y}$ in the infinite strip
$\left[0,1\right]\times\mathbb{R},$ the square $\left[0,1\right]^{2},$
and in $\left[0,1\right]\times C,$ for certain fractal sets $C$. 

More precisely, we consider operators in the Hilbert space $L^{2}\left([0,1]^{2}\right)$
of square integrable functions $\left[0,1\right]^{2}\to\mathbb{C}$
equipped with the inner product 
\[
\left\langle f,g\right\rangle =\int_{0}^{1}\int_{0}^{1}\overline{f(x,y)}g(x,y)dxdy.
\]
The operator $P_{\min}=\frac{1}{i2\pi}\partial_{x}$ with domain $\mathrm{dom}\left(P_{\min}\right)=C_{c}^{\infty}\left([0,1]^{2}\right)$
is \emph{symmetric}, that is 
\[
\left\langle P_{\min}f,g\right\rangle =\left\langle f,P_{\min}g\right\rangle 
\]
for all $f,g$ in $C_{c}^{\infty}\left([0,1]^{2}\right).$ The adjoint
of $P_{\min}$ is $P_{\min}^{*}=P_{\mathrm{max}}=\frac{1}{i2\pi}\partial_{x},$
acting in the distribution sense, with domain
\[
\mathrm{dom}\left(P_{\mathrm{max}}\right)=\left\{ f\in L^{2}\left(\left[0,1\right]^{2}\right)\mid\partial_{x}f\in L^{2}\left(\left[0,1\right]^{2}\right)\right\} .
\]
Hence, for $k=0,1,$ the map $f\to f(k,\cdot)$ maps $\mathrm{dom}\left(P_{\mathrm{max}}\right)$
onto $L^{2}\left([0,1]\right)$ and 
\[
\left\langle P_{\mathrm{max}}f,g\right\rangle -\left\langle f,P_{\mathrm{max}}g\right\rangle =i2\pi\left(\int_{0}^{1}\overline{f(1,y)}g(1,y)dy-\int_{0}^{1}\overline{f(0,y)}g(0,y)dy\right),
\]
for all $f$ and $g$ in the domain of $P_{\mathrm{max}}.$ The selfadjoint
extensions of $P_{\min}$ are in one-to-one correspondence with the
unitary operators $V$ acting in $L^{2}([0,1]):$ The selfadjoint
extension $P_{V}$ of $P_{\min}$ corresponding to the unitary $V$
is the restriction of $P_{\mathrm{max}}$ whose domain $\mathrm{dom}\left(P_{V}\right)$
is the functions $f$ in $\mathrm{dom}\left(P_{\mathrm{max}}\right)$
that satisfies the boundary condition 
\begin{equation}
f(1,\cdot)=Vf(0,\cdot).\label{sec-2-eq:x-boundary-condition}
\end{equation}
We will call $V$ a \emph{boundary} unitary. 

In the interval case $P_{\theta}$ and $P_{\theta'}$ are unitary
equivalent if and only if $\theta=\theta'.$ The analogue for the
unit square is: 
\begin{thm}
\label{thm:1.1}Let $U,V:L^{2}([0,1])\rightarrow L^{2}([0,1])$ be
two boundary unitary operators. Then the corresponding selfadjoint
extensions $P_{U}$ and $P_{V}$ of $P_{\min}$ are unitary equivalent
if and only if $U$ and $V$ are unitary equivalent. 
\end{thm}
In the case of the interval the spectrum $P_{\theta}$ is discrete,
in fact equal to $\theta+\mathbb{Z}.$ But for the square the spectrum
of $P_{V}$ has a much richer structure. Theorem \ref{Sec-4-thm:Spectral-thm-for-P}
gives a description of the spectral measure associated with $P_{V}$
in terms of the spectral measure associated with $V.$ Theorems \ref{thm:1.2}
and \ref{thm:1.3} are consequences of Theorem \ref{Sec-4-thm:Spectral-thm-for-P}. 
\begin{thm}
\label{thm:1.2}Let $P_{V}$ be a selfadjoint extension of $P_{\min}$
associated with the boundary unitary operator $V:L^{2}([0,1])\rightarrow L^{2}([0,1])$.
Then $P_{V}+1$ is unitary equivalent to $P_{V}$.
\end{thm}
A consequence of the following result is that, in contrast to the
interval case, the spectrum of $P_{V}$ need not be discrete. 
\begin{thm}
\label{thm:1.3}The spectrum of $P_{V}$ equals the set of $\lambda$
for which $e\left(\lambda\right)$ is in the spectrum of $V.$ 
\end{thm}
A pair $\left(\mu,\nu\right)$ of measures on $\mathbb{R}^{d}$ is
called a \emph{spectral pair}, if $F:f\to\widehat{f}\left(\lambda\right):=\int f(x)e\left(-\lambda x\right)d\mu(x)$
determines a unitary $F:L^{2}\left(\mu\right)\to L^{2}\left(\nu\right).$
In this form the notion was introduced in \cite{JP99}. The case where
$\mu$ is a the restriction of Lebesgue measure to a measurable set
$\Omega$ was studied in \cite{Ped87}, in this case the set $\Omega$
is called a \emph{spectral set}, provided $\left(\mu,\nu\right)$
is a spectral pair for some measure $\nu.$ The notion was introduced
in \cite{Fu74} in the case where $\Omega$ has finite Lebesgue measure.
A connected open set in $\mathbb{R}^{d}$ is a spectral set if and
only if there are commuting selfadjoint extensions of the partial
derivatives $\tfrac{1}{i2\pi}\partial_{x_{k}}\Big|_{C_{c}^{\infty}\left(\Omega\right)},$
$k=1,\ldots,d,$ in $L^{2}\left(\Omega\right).$ In the affirmative
case, the support of $\nu$ is the joint spectrum of the commuting
selfadjoint extensions. See, \cite{Fu74}, \cite{Jo82}, and \cite{Ped87}
for proofs of these claims. We use Theorem \ref{Sec-4-thm:Spectral-thm-for-P}
to characterize the boundary unitary operators that lead to commuting
extensions of the partial derivatives and we calculate the joint spectra
for the infinite strip, the unit square, and a fractal domain. This
was previously done for the unit square, by a different method, in
\cite{JP00}. 

While our primary interest is in the unit square, we find it convenient
to establish many of our results in a more abstract setting. This
also allows us to apply our techniques to the study of spectral sets.
Using 
\[
L^{2}\left(\left[0,1\right]^{2}\right)=L^{2}\left(\left[0,1\right]\right)\otimes L^{2}\left(\left[0,1\right]\right),
\]
we replace the second $L^{2}-$space by a generic Hilbert space and
we replace the interval $[0,1]$ by the generic interval $[\alpha,\beta].$ 

Fix real numbers $\alpha<\beta$ and a Hilbert space $H.$ Consider
the Hilbert space 
\[
\mathscr{H}:=L^{2}\left([\alpha,\beta],H\right)=L^{2}\left(\left[\alpha,\beta\right]\right)\otimes H
\]
of $L^{2}-$functions $\left[\alpha,\beta\right]\to H$ equipped with
the inner product 
\[
\left\langle f\mid g\right\rangle :=\int_{\alpha}^{\beta}\left\langle f(x)\mid g(x)\right\rangle dx,
\]
where $\left\langle f(x)\mid g(x)\right\rangle $ is the inner product
in $H.$ We will consider selfadjoint extensions of the operator $P_{0}$
determined by
\[
P_{0}f:=\tfrac{1}{i2\pi}f'
\]
with the domain
\[
\mathrm{dom}\left(P_{0}\right):=\left\{ f\in L^{2}\left([\alpha,\beta],H\right)\mid f'\in L^{2}\left([\alpha,\beta],H\right),f(\alpha)=f(\beta)=0\right\} .
\]
The selfadjoint extensions of $P_{0}$ are determined by boundary
conditions, more precisely, they are parametrized by the unitary operators
$U:H\to H.$ The selfadjoint extension $P_{U}$ corresponding to the
unitary $U$ is 
\begin{equation}
P_{U}f=\tfrac{1}{i2\pi}f'\label{eq:P_U}
\end{equation}
for 
\begin{equation}
f\in\mathrm{dom}\left(P_{U}\right):=\left\{ f\in L^{2}\left(\left[\alpha,\beta\right],H\right)\mid f'\in L^{2}\left(\left[\alpha,\beta\right],H\right),f(\beta)=Uf(\alpha)\right\} .\label{eq:dom-P_U}
\end{equation}
For more details on this correspondence, see Appendix \ref{sec-A-:Selfadjoint-Extensions}. 

Section \ref{sec-2-:Unitary Group} contains a formula for the unitary
group $e\left(aP_{U}\right).$ This formula is used to prove Theorem
\ref{thm:1.1}. In Section \ref{sec-3:Eigenvalues} we discuss eigenvalues
and eigenvectors of $P_{U}$ and we present some natural examples
where $P_{U}$ has a complete set of eigenvectors. In Section \ref{sec-4:A-Spectral-Theorem}
we establish the connection, Theorem \ref{Sec-4-thm:Spectral-thm-for-P},
between the projection-valued measures of $P_{U}$ and of $U.$ We
use this connection to establish Theorems \ref{thm:1.2} and \ref{thm:1.3}.
Section \ref{Sec-5-sec:Spectral-Pairs} includes the analysis of spectral
pairs discussed above. Appendix \ref{sec-A-:Selfadjoint-Extensions}
contains the details needed to establish that the collection $P_{U},$
$U$ unitary in $H$ is the collection of all self adjoint extensions
of $P_{0}.$ Appendix \ref{sec-B:Questions} contains some open problems. 

The papers \cite{Hir00} and \cite{Ree88} contains discussions of
momentum operators in the complements of simple compact sets. In fact,
any paper discussing the canonical commutation relations in proper
subsets of $\mathbb{R}^{d},$ for $d>1$ contains, at least implicitly,
material related to the present paper. For other recent work on momentum
operators we refer to \cite{Car99}, \cite{ES10}, \cite{Ex12}, \cite{FKW07},
\cite{JPT11-1}, \cite{JPT11-2},\cite{JPT12}, and \cite{JPT12a}. 

This paper is based on standard operator theory, the needed background
can be found in \cite{ReSi81,ReSi75}. Some recent text books containing
most, but not all, of what we need are \cite{dO09} and \cite{Gru09}.

\section{The Unitary Group\label{sec-2-:Unitary Group}}

Fix $\alpha<\beta.$ For a real number $r,$ let $\alpha\leq\left\langle r\right\rangle <\beta$
and $\left\lfloor r\right\rfloor \in\mathbb{Z}$ be such that $r=\left\langle r\right\rangle +\left\lfloor r\right\rfloor (\beta-\alpha).$
Note, $\left\langle r\right\rangle $ and $\left\lfloor r\right\rfloor $
are uniquely determined by these conditions. For a fixed real number
$a$, the transformation $\tau_{a}:x\to\left\langle x+a\right\rangle $
is a measure preserving transformation of $[\alpha,\beta],$ hence
$T_{a}f:=f\circ\tau_{a}$ is a unitary in $L^{2}([\alpha,\beta]).$ 

Corresponding to any selfadjoint extension $P_{V},$ there is strongly
continuous unitary one-parameter group $a\to e\left(aP_{V}\right),$
where
\[
e(r):=e^{i2\pi r}.
\]
The unitary group can, for example, be determined from $P_{V}$ by
an application of the spectral theorem. The result below establishes
an explicit formula for the action of $e\left(aP_{V}\right).$ Conversely,
$P_{V}$ can be obtained from $e\left(aP_{V}\right)$ by differentiating
$a\to e\left(aP_{V}\right)$ at $a=0.$ 
\begin{prop}
\label{sec-2-prop:formula-unitary-group} Let $V:H\to H$ be a unitary
and let $P_{V}$ be the corresponding selfadjoint extension of $P_{0}$
determined by (\ref{eq:P_U}) and (\ref{eq:dom-P_U}). The unitary
group $a\to e\left(aP_{V}\right)$ satisfies
\begin{equation}
e\left(aP_{V}\right)f(x)=\left(T_{a}\otimes V^{\left\lfloor x+a\right\rfloor }\right)f(x)=V^{\left\lfloor x+a\right\rfloor }f(\left\langle x+a\right\rangle )\label{sec-2-eq:unitary-group-formula}
\end{equation}
 for all $f$ in $L^{2}\left(\left[\alpha,\beta\right]\right)\otimes H,$
a.e. $x$ in $\left[\alpha,\beta\right],$ and all $a$ in $\mathbb{R}.$ \end{prop}
\begin{proof}
This is well know, we include a proof for completeness. Let $U_{a}:=T_{a}\otimes V^{\left\lfloor x+a\right\rfloor }.$
We must show that $e(aP_{V})=U_{a}.$ We begin by checking that $U_{a}$
is a strongly continuous unitary one-parameter group. 

Let $I$ denote the identity in $L^{2}\left(\left[\alpha,\beta\right]\right).$
Then $I\otimes V^{\left\lfloor x+a\right\rfloor }$ is unitary because
\[
I\otimes V^{\left\lfloor x+a\right\rfloor }=\left(\begin{array}{cc}
I\otimes V^{\left\lfloor a\right\rfloor } & 0\\
0 & I\otimes V^{1+\left\lfloor a\right\rfloor }
\end{array}\right)
\]
with respect to the decomposition 
\[
L^{2}\left([0,1]\right)\otimes H=\left(L^{2}\left([\alpha,\beta-\left\langle a\right\rangle \right)\otimes H\right)\oplus\left(L^{2}\left([\beta-\left\langle a\right\rangle ,\beta\right)\otimes H\right)
\]
of $L^{2}\left([0,1]\right)\otimes H.$ Hence $U_{a}$ is unitary,
since $T_{a}$ is unitary. 

Next, we check that $U_{a},$ $a\in\mathbb{R}$ is a group action,
that is that $U_{a}U_{b}=U_{a+b},$ for all $a$ and $b$ in $\mathbb{R}.$
Consider $f(x)=g(x)h,$ where $g\in L^{2}\left(\left[\alpha,\beta\right]\right)$
and $h\in H,$ then 
\[
I\otimes V^{\left\lfloor x+a\right\rfloor }f(x)=g(x)V^{\left\lfloor x+a\right\rfloor }h.
\]
Hence, 
\[
U_{a}f(x)=g\left(\left\langle x+a\right\rangle \right)V^{\left\lfloor x+a\right\rfloor }h
\]
so 
\begin{eqnarray*}
U_{a}\left(U_{b}f\right)(x) & = & U_{b}\left(g\left(\left\langle x+a\right\rangle \right)\left(V^{\left\lfloor x+a\right\rfloor }h\right)\right)\\
 & = & g\left(\left\langle \left\langle x+a\right\rangle +b\right\rangle \right)V^{\left\lfloor \left\langle x+a\right\rangle +b\right\rfloor }\left(V^{\left\lfloor x+a\right\rfloor }h\right).
\end{eqnarray*}
Now $\left\langle \left\langle x+a\right\rangle +b\right\rangle =\left\langle x+a+b\right\rangle $
and $\left\lfloor \left\langle x+a\right\rangle +b\right\rfloor +\left\lfloor x+a\right\rfloor =\left\lfloor x+a+b\right\rfloor $,
hence $U_{a}U_{b}=U_{a+b}$ for all simple tensors $f(x)=g(x)h$ and
therefore for all $f$ in $L^{2}\left(\left[\alpha,\beta\right]\right)\otimes H.$ 

To see that $a\to U_{a}f$ is continuous. Consider $f\in L^{2}\left([\alpha,\beta-b)\right)\otimes H$
for some $0<b<\beta-\alpha.$ Then 
\[
U_{a}f(x)=f(x+a)
\]
for all $a<b.$ Hence $U_{a}f\to f$ as $a\to0.$ Since, $\bigcup_{0<b<\beta-\alpha}L^{2}\left([\alpha,\beta-b)\right)\otimes H$
is dense in $L^{2}\left(\left[\alpha,\beta\right]\right)\otimes H,$
we conclude that $U_{a}$ is strongly continuous. 

Since $U_{a}$ is a strongly continuous unitary group on $L^{2}\left(\left[\alpha,\beta\right]\right)\otimes H,$
there is a selfadjoint operator $Q$ on $L^{2}\left(\left[\alpha,\beta\right]\right)\otimes H$
such that $e(aQ)=U_{a}$ for all $a\in\mathbb{R}.$ Now 
\[
\tfrac{1}{i2\pi}Qf=\lim_{a\to0}\tfrac{1}{i2\pi a}\left(U_{a}f-f\right)
\]
and the domain of $Q$ is the set of all $f$ for which the limit
exists. If $f:[\alpha,\beta]\to H$ is compactly supported in $(\alpha,\beta)$,
then for sufficiently small $a$ we have 
\[
\tfrac{1}{i2\pi a}\left(U_{a}f-f\right)(x)=\tfrac{1}{i2\pi a}\left(f(x+a)-f(x)\right),
\]
by definition of $U_{a}.$ Consequently, $P_{0}\subset Q,$ meaning
that $P_{0}$ is a restriction of $Q.$ Taking the adjoint gives $Q\subset P_{0}^{*}.$
Hence, if $f\in\mathrm{dom}(Q),$ then $f\in\mathrm{dom}\left(P_{0}^{*}\right),$
hence, by Lemma \ref{sec-A-lem:the-adjoint-of-P}, $f'(x)$ exists
and $Qf(x)=\tfrac{1}{i2\pi}f'(x)$ for a.e. $x$ in $(\alpha,\beta).$
Let $0<a<\beta-\alpha.$ then we have 
\begin{equation}
\tfrac{1}{i2\pi a}\left(U_{a}f-f\right)(x)=\tfrac{1}{i2\pi a}\left(Vf(\left\langle x+a\right\rangle )-f(x)\right)\label{sec-2-eq:QP}
\end{equation}
when $\beta-a<x<\beta,$ by definition of $U_{a}.$ As $a\to0,$ we
have $x\to\beta$, $\left\langle x+a\right\rangle \to\alpha,$ so
$f(x)\to f(\beta),$ and $Vf(\left\langle x+a\right\rangle )\to Vf(\alpha).$
Thus, (\ref{sec-2-eq:QP}) implies $f(\beta)=Vf(\alpha).$ It follows
that $\mathrm{dom}(Q)\subseteq\mathrm{dom}\left(P_{V}\right).$ Consequently,
$Q=P_{V}$ as we needed to show. \end{proof}
\begin{lem}
\label{sec-2-lem:expH}The unitary group $e\left(aP_{V}\right)$ in
$L^{2}\left(\left[\alpha,\beta\right]\right)\otimes H$ and the boundary
unitary $V$ in $H$ are related by 
\begin{equation}
e\left(\left(\beta-\alpha\right)P_{V}\right)=I\otimes V,\label{sec-2-eq:vb}
\end{equation}
where $I$ is the identity operator in $L^{2}\left(\left[\alpha,\beta\right]\right).$\end{lem}
\begin{proof}
Set $a=\beta-\alpha$ in (\ref{sec-2-eq:unitary-group-formula}) and
use that $T_{\beta-\alpha}=I.$ 
\end{proof}
Theorem \ref{thm:1.1} is a consequence of 
\begin{thm}
\label{Sec-2-thm:Unitary-equivalence}Let $U,V:H\rightarrow H$ be
two boundary unitary operators. The corresponding selfadjoint extensions
$P_{U}$ and $P_{V}$ of $P_{0},$ in $L^{2}\left(\left[\alpha,\beta\right]\right)\otimes H,$
determined by (\ref{eq:P_U}) and (\ref{eq:dom-P_U}) are unitary
equivalent if and only if $U$ and $V$ are unitary equivalent. \end{thm}
\begin{proof}
Suppose $P_{U}$ and $P_{V}$ are unitary equivalent. Let $W$ be
a unitary such that $P_{U}=W^{*}P_{V}W$, then it follows from the
spectral theorem that
\begin{equation}
We\left(aH_{U}\right)=e\left(aH_{V}\right)W\label{sec-2-eq:W-e}
\end{equation}
for all real numbers $a$. Setting $a=\beta-\alpha$ in (\ref{Sec-2-eq:W})
and using Lemma \ref{sec-2-lem:expH} leads to
\begin{equation}
W\left(I\otimes U\right)=\left(I\otimes V\right)W.\label{eq:W-tensor}
\end{equation}
Where $I$ is the identity operator acting in $L^{2}\left(\left[\alpha,\beta\right]\right).$
By Proposition \ref{sec-2-prop:formula-unitary-group}, equation (\ref{sec-2-eq:W-e})
takes the form 
\begin{align}
W\left(T_{a}\otimes U^{\left\lfloor x+a\right\rfloor }\right) & =\left(T_{a}\otimes V^{\left\lfloor x+a\right\rfloor }\right)W\label{Sec-2-eq:W}
\end{align}
for all real numbers $a.$ Combining (\ref{Sec-2-eq:W}) and (\ref{eq:W-tensor})
we have 
\begin{align*}
\left(I\otimes V^{\left\lfloor x+a\right\rfloor }\right)W\left(T_{a}\otimes I_{H}\right) & =\left(I\otimes V^{\left\lfloor x+a\right\rfloor }\right)\left(T_{a}\otimes I_{H}\right)W,
\end{align*}
where $I_{H}$ is the identity in $H.$ Therefore 
\begin{align}
W\left(T_{a}\otimes I_{H}\right) & =\left(T_{a}\otimes I_{H}\right)W.\label{eq:W-T-tensor}
\end{align}
Let $e_{m}(x)=e(mx)$ for $x,m\in\mathbb{R}.$ Applying (\ref{eq:W-T-tensor})
to $f=e_{m}\otimes h,$ $m\in\mathbb{R},$ $h\in H$, we get 
\[
e(ma)W\left(e_{m}\otimes h\right)(x)=\left(W\left(e_{m}\otimes h\right)\right)(\left\langle x+a\right\rangle )
\]
for all $a.$ Consequently, there are $h_{m}$ in $H$ such that 
\begin{equation}
W\left(e_{m}\otimes h\right)=e_{m}\otimes h_{m}\label{eq:WWWW}
\end{equation}
for all $m.$ Let $P$ be the projection in $L^{2}\left([\alpha,\beta]\right)\otimes H$
onto the functions that are independent of $x$. Setting $m=0$ in
(\ref{eq:WWWW}) shows that the range of $P$ is invariant under $W.$
Hence,
\begin{eqnarray}
WP & = & PWP.\label{Sec-2-eq:WP=00003DPWP}
\end{eqnarray}
Taking the adjoint of (\ref{eq:W-T-tensor}) and repeating this argument
shows that 
\begin{equation}
W^{*}P=PW^{*}P.\label{eq:W*P=00003DPW*P}
\end{equation}
Combining (\ref{Sec-2-eq:WP=00003DPWP}) and (\ref{eq:W*P=00003DPW*P})
we get 
\begin{equation}
WP=PW\quad\text{and}\quad W^{*}P=PW^{*}.\label{eq:WP=00003DPW}
\end{equation}
Let $i:H\rightarrow L^{2}\left([\alpha,\beta]\right)\otimes H$ be
the isometric embedding determined by 
\[
\left(ig\right)\left(x\right)=\left(\beta-\alpha\right)^{-1/2}g,
\]
for all $x$ in $[\alpha,\beta].$ Then
\[
i^{*}f=\left(\beta-\alpha\right)^{-1/2}\int_{\alpha}^{\beta}f(x)\, dx
\]
for $f\in L^{2}\left([\alpha,\beta]\right)\otimes H.$ And, if $B:H\to H$
is a bounded operator, then 
\begin{equation}
\left(I\otimes B\right)i=iB.\label{eq:1}
\end{equation}
Replacing $B$ by $B^{*}$ in (\ref{eq:1}) and taking the adjoint
yields
\begin{equation}
i^{*}\left(I\otimes B\right)=Bi^{*}.\label{Sec-2-eq:2}
\end{equation}
Consequently, 
\begin{eqnarray*}
\left(i^{*}Wi\right)U & = & i^{*}W\left(iU\right)\\
 & = & i^{*}W\left(I\otimes U\right)i\,\,\,\,\,\,\,\,\,\,\,\,\,\,\,(\mbox{by }(\ref{eq:1}))\\
 & = & i^{*}\left(I\otimes V\right)Wi\,\,\,\,\,\,\,\,\,\,\,\,\,\,\,(\mbox{by }(\ref{eq:W-tensor}))\\
 & = & V\left(i^{*}Wi\right).\,\,\,\,\,\,\,\,\,\,\,\,\,\,\,\,\,\,\,\,\,\,\,\,(\mbox{by }(\ref{Sec-2-eq:2}))
\end{eqnarray*}
It remains to show $i^{*}Wi$ is unitary. It is easy to see that
$P=ii^{*}$ and $i^{*}i=I_{H}.$ Recall, $I_{H}$ is the identity
in $H.$ 

Using (\ref{eq:WP=00003DPW}), $i^{*}i=P,$ and $Pi=i$ simple calculations
show that 
\[
\left(i^{*}Wi\right)^{*}\left(i^{*}Wi\right)=I_{H}
\]
and 
\[
\left(i^{*}Wi\right)\left(i^{*}Wi\right)^{*}=I_{H}.
\]
Therefore, $i^{*}Wi$ is unitary, and $U$ is unitary equivalent to
$V$.

Conversely, suppose $U$ is unitary equivalent to $V.$ Let $W$ be
a unitary in $H$ such that $WU=VW.$ Then, by Proposition \ref{sec-2-prop:formula-unitary-group}
\begin{eqnarray*}
\left(I\otimes W\right)e\left(aP_{U}\right)f(x) & = & \left(I\otimes W\right)\left(T_{a}\otimes U^{\left\lfloor x+a\right\rfloor }\right)f(x)\\
 & = & \left(T_{a}\otimes V^{\left\lfloor x+a\right\rfloor }\right)\left(I\otimes W\right)f(x)\\
 & = & e\left(aP_{V}\right)\left(I\otimes W\right)f(x),
\end{eqnarray*}
for all $a\in\mathbb{R}.$ Consequently, $\left(I\otimes W\right)P_{U}=P_{V}\left(I\otimes W\right).$ 
\end{proof}

\section{Eigenvalues\label{sec-3:Eigenvalues}}

The first result in this section establishes a relationship between
the eigenvalues and eigenvectors of $V$ and the eigenvalues and eigenvectors
of $P_{V}$. We extend this result to include continuous spectrum
in Section \ref{sec-4:A-Spectral-Theorem}, see Theorem \ref{Sec-4-thm:Spectral-thm-for-P}.
\begin{prop}
\label{Sec-3-prop:Eigenvalues-P_V}Let $V:H\to H$ be a unitary and
let $P_{V}$ be the corresponding selfadjoint extension of $P_{0}$,
in $L^{2}\left([\alpha,\beta],H\right)$, determined by (\ref{eq:P_U})
and (\ref{eq:dom-P_U}). Then $\lambda$ is an eigenvalues of $P_{V}$
if and only if $e\left(\left(\beta-\alpha\right)\lambda\right)$ is
an eigenvalue of $V.$ In particular, if $\lambda$ is an eigenvalue
for $P_{V},$ so is $\lambda+\tfrac{m}{\beta-\alpha},$ for any integer
$m.$ Furthermore, $h_{j},$ $1\leq j<n+1$ is an orthogonal basis
for the eigenspace of $V$ corresponding to the eigenvalue $e\left(\left(\beta-\alpha\right)\lambda\right)$
if and only if $f_{j}(x)=e(\lambda x)h_{j}$ is an orthogonal basis
for the eigenspace of $P_{V}$ corresponding to the eigenvalue $\lambda.$ \end{prop}
\begin{proof}
Suppose $e\left(\left(\beta-\alpha\right)\lambda\right)$ is an eigenvalue
of $V$ and $h\in H$ is a corresponding eigenvector. Let 
\[
f(x):=e(\lambda x)h.
\]
Then, $Vf(\alpha)=e(\lambda\alpha)Vh=e(\lambda\beta)h=f(\beta),$
hence $f$ is in the domain of $P_{V}.$ Since 
\[
P_{V}f(x)=\frac{1}{i2\pi}\partial_{x}f(x)=\frac{1}{i2\pi}\partial_{x}e(\lambda x)h=\lambda e(\lambda x)h=\lambda f(x)
\]
we conclude that $\lambda$ is an eigenvalue of $P_{V}$

Conversely, suppose $\lambda$ is an eigenvalue for $P_{V}$ and $f$
is a corresponding eigenvector. Then $P_{V}f=\lambda f$ implies $\partial_{x}f=2\pi i\lambda f.$
Solving this differential equation gives $f(x)=e(\lambda x)h$ for
some $h\in H.$ Since $f$ is in the domain of $P_{V},$ it follows
from (\ref{eq:dom-P_U}) that $Vh=e\left(\left(\beta-\alpha\right)\lambda\right)h$.
Consequently, $e\left(\left(\beta-\alpha\right)\lambda\right)$ is
an eigenvalue for $V.$ 

We leave the details of the eigenvector claims to the reader. 
\end{proof}
In the remainder of this section we consider the case where $\alpha=0,$
$\beta=1,$ and $H=L^{2}\left(\left[0,1\right]\right).$ Hence $\mathscr{H}=L^{2}\left(\left[0,1\right]^{2}\right).$
For a measure preserving transformation $v:[0,1]\to[0,1]$ and a measurable
function $\theta:[0,1]\to\mathbb{R}$, let $V=V_{v,\theta}$ be the
unitary on $L^{2}\left([0,1]\right)$ determined by
\begin{equation}
Vg(y)=e\left(\theta(y)\right)g\left(v(y)\right).\label{Sec-3-eq:general-geometric-boundary}
\end{equation}
Let $v^{0}(y)=y$ and inductively $v^{n}=v\circ v^{n-1}$ for $n>0.$
If 
\[
\phi_{a}(x,y)=\left(\left\langle a+x\right\rangle ,v^{\left\lfloor x+a\right\rfloor }y\right),
\]
then it follows from (\ref{sec-2-eq:unitary-group-formula}) that

\begin{eqnarray*}
e\left(aP_{V}\right)f(x,y) & = & e\left(\left\lfloor x+a\right\rfloor \theta(y)\right)f\left(\left\langle a+x\right\rangle ,v^{\left\lfloor x+a\right\rfloor }(y)\right)\\
 & = & e\left(\left\lfloor x+a\right\rfloor \theta(y)\right)f\circ\phi_{a}(x,y)
\end{eqnarray*}
for $f$ in $L^{2}\left([0,1]^{2}\right).$ 
\begin{rem}[Geometric Boundary Conditions]
\label{Sec-3-remark:Geometric-Boundary-Condition}In the case of
the unit interval $[0,1],$ the selfadjoint momentum operators are
determined by the boundary condition $f(1)=e(\theta)f(0).$ Geometrically,
we can think of this as identifying the endpoints up to a phase shift.
A natural analogue of this for the unit square is the special case
$Vg(y)=e(\theta)g\left(\left\langle y+r\right\rangle \right),$ $r,\theta\in\mathbb{R},$
of (\ref{Sec-3-eq:general-geometric-boundary}), in this case the
spectrum of $V$ is well understood, see Example \ref{Sec-3-Example:rotations}.
A more general analogue of the interval case is $Vg(y)=e\left(\theta(y)\right)g\left(\left\langle y+r\right\rangle \right),$
for some measurable $\theta:[0,1]\to\mathbb{R}.$ In this case, the
spectral type of $V$ is pure and the multiplicity is uniform, see
\cite{Hel86}. The exact spectral type depends on the function $\theta,$
see, for example, \cite{ILM99} and the references therein. 
\end{rem}
We have the following corollary to Proposition \ref{Sec-3-prop:Eigenvalues-P_V}.
\begin{cor}
\label{Sec-3-cor:ergodic}$v$ is an ergodic transformation on $[0,1]$
if and only if $\phi_{a}$ is an ergodic action of $\mathbb{R}$ on
$[0,1]^{2}.$\end{cor}
\begin{proof}
Let $\theta(y)=0$ for all $y.$ It follows from Proposition \ref{Sec-3-prop:Eigenvalues-P_V}
that $1$ is an eigenvalue for $V$ with multiplicity one if and only
if $1$ is an eigenvalue for $e\left(aP_{V}\right)$ with multiplicity
one.
\end{proof}
A special case of Proposition \ref{Sec-3-prop:Eigenvalues-P_V} is:
\begin{cor}
\label{Sec-3-cor:Diagonal-V}Fix $r_{n}$ in $[0,1[.$ Let $V$ be
determined by $Ve_{n}=e(r_{n})e_{n},$ then the set 
\[
\bigcup_{n\in\mathbb{Z}}\left(r_{n}+\mathbb{Z}\right)
\]
equals the set of eigenvalues for $P_{V}.$ 
\end{cor}
Rotations provide a natural class of examples for the Corollary \ref{Sec-3-cor:Diagonal-V}
and Corollary \ref{Sec-3-cor:ergodic}:
\begin{example}[Rotations]
\label{Sec-3-Example:rotations} Let $0\leq r<1$ be a real number.
Consider 
\[
\left(Vg\right)(y)=g\left(v(y)\right)
\]
where $v(y)=\left\langle y+r\right\rangle $ is the fractional part
of $y+r.$ Using Fourier series, 
\[
g(y)=\sum_{n\in\mathbb{Z}}\widehat{g}(n)e(ny)
\]
where $\widehat{g}(n)=\int_{0}^{1}g(y)e(-ny)dy,$ it follows that
\[
V\sum_{n\in\mathbb{Z}}\widehat{g}(n)e(ny)=\sum_{n\in\mathbb{Z}}\widehat{g}(n)e(nr)e(ny).
\]
In particular, $e_{n}(y)=e(ny)$ is an eigenfunction for $V$ corresponding
to the eigenvalue $e(nr).$ So, by Proposition \ref{Sec-3-prop:Eigenvalues-P_V},
the set of eigenvalues for $P_{V}$ is the set $r\mathbb{Z}+\mathbb{Z}=\left\{ ra+b\mid a,b\in\mathbb{Z}\right\} .$
Compared to the previous example we have $r_{n}=\left\langle nr\right\rangle ,$
the fractional part of $nr.$ This is used in Remark \ref{Sec-5:Remark-Square-geometric}.

If $r$ is irrational, then $v$ is ergodic. Clearly, $r_{m}\neq r_{n}$
for all $m\neq n,$ so each eigenvalue for $P_{V}$ has multiplicity
one. Furthermore, it is well known that the sequence $r_{n}$ is uniformly
distributed in the interval $[0,1].$ See e.g., \cite{KN74}.

If $r$ is rational, the set $\left\{ r_{n}\mid n\in\mathbb{Z}\right\} $
is finite and each eigenvalue of $P_{V}$ has infinite multiplicity. 
\begin{example}
If $V_{k}g=g\circ v_{k},$ where $v_{1}(y)=1-y$ and $v_{2}(y)=\tau_{1/2}(y)=\left\langle y+1/2\right\rangle $
is the fractional part of $y+1/2,$ then $V_{1}$ and $V_{2}$ are
unitary equivalent. Hence, $P_{V_{1}}$ and $P_{V_{2}}$ are unitary
equivalent, by Theorem \ref{Sec-2-thm:Unitary-equivalence}. However,
dynamically $v_{1}$ is a reflection and $v_{2}$ is a translation.
\textcolor{red}{}
\end{example}
\end{example}

\section{A Spectral Theorem\label{sec-4:A-Spectral-Theorem}}

In this section we obtain a formula for the spectral resolution $P_{V}$
in terms of the the spectral resolution of the boundary unitary $V.$
This is essentially contained in Section \ref{sec-3:Eigenvalues}
for the set of eigenvalues. We begin working toward the spectral representation
of $P_{V}$, when there is continuous spectrum, by finding the Green's
function and using it to find a formula for the resolvent of $P_{V}.$ 
\begin{prop}
\label{Sec-4-prop:green}Consider a boundary unitary $V:H\rightarrow H$
and the corresponding selfadjoint extension $P_{V}$ of $P_{0}$,
in $L^{2}\left(\left[\alpha,\beta\right]\right)\otimes H,$ determined
by (\ref{eq:P_U}) and (\ref{eq:dom-P_U}). For all $f\in L^{2}\left(\left[\alpha,\beta\right],H\right)$,
and $z\in\mathbb{C}\backslash\mathbb{R}$, 
\begin{equation}
\left(z-P_{V}\right)^{-1}f\left(x,\cdot\right)=\int_{\alpha}^{\beta}G\left(x,s,z\right)f\left(s,\cdot\right)ds\label{Sec-4-eq:resv}
\end{equation}
where the Green's function $G\left(x,s,z\right)$, is given by 
\begin{equation}
G\left(x,s,z\right)=\begin{cases}
i2\pi\left(\left(1-Ve_{\beta-\alpha}\left(-z\right)\right)^{-1}-1\right)e\left(z\left(x-s\right)\right) & \alpha\leq s<x\leq\beta\\
i2\pi\left(1-Ve_{\beta-\alpha}\left(-z\right)\right)^{-1}e\left(z\left(x-s\right)\right) & \alpha\leq x<s\leq\beta
\end{cases}\label{eq:gr}
\end{equation}
for all $z\in\mathbb{C}\backslash\mathbb{R}$. Where $e_{\lambda}(z):=e(\lambda z)=e^{i2\pi\lambda z}.$ \end{prop}
\begin{proof}
Let $f\in L^{2}\left(\left[\alpha,\beta\right],H\right)$, $z\in\mathbb{C}\backslash\mathbb{R}$,
and let 
\[
g:=\left(z-P_{V}\right){}^{-1}f\in\mathrm{dom}\left(P_{V}\right).
\]
That is, $g$ is the unique solution to the differential equation
\begin{equation}
zg\left(x\right)-\frac{1}{2\pi i}g'\left(x\right)=f\left(x\right)\label{eq:ode}
\end{equation}
satisfying the boundary condition 
\begin{equation}
Vg\left(\alpha\right)=g\left(\beta\right).\label{Sec-4-eq:bd1}
\end{equation}
Multiply both sides of (\ref{eq:ode}) by the integrating factor $e\left(-zx\right)$,
we get 
\[
\frac{d}{dx}\left(e\left(-zx\right)g\left(x\right)\right)=-i2\pi\, e\left(-zx\right)f\left(x\right)
\]
so that 
\begin{equation}
g\left(x\right)=e\left(z\left(x-\alpha\right)\right)g\left(\alpha\right)-i2\pi\int_{\alpha}^{x}e\left(z\left(x-s\right)\right)f\left(s\right)ds.\label{eq:rf}
\end{equation}
By the boundary condition (\ref{Sec-4-eq:bd1}), 
\[
\left(V-e_{\beta-\alpha}\left(z\right)\right)g\left(\alpha\right)=-i2\pi\int_{\alpha}^{\beta}e\left(z\left(\beta-s\right)\right)f\left(s\right)ds.
\]
Note that $g\left(\alpha\right)$ has a unique solution if and only
if $e_{\beta-\alpha}\left(z\right)\notin sp\left(V\right)$. In that
case, 
\begin{equation}
g\left(\alpha\right)=i2\pi\left(e_{\beta-\alpha}\left(z\right)-V\right)^{-1}\int_{\alpha}^{\beta}e\left(z\left(\beta-s\right)\right)f\left(s\right)ds.\label{Sec-4-eq:g-of-alpha}
\end{equation}
Substitute $g\left(\alpha\right)$ into (\ref{eq:rf}), and it follows
that 
\begin{eqnarray*}
\left(z-P_{V}\right)^{-1}f\left(x\right) & = & \int_{\alpha}^{x}i2\pi\left(\left(1-Ve_{\beta-\alpha}\left(-z\right)\right)^{-1}-1\right)e\left(z\left(x-s\right)\right)f\left(s\right)ds\\
 & + & \int_{x}^{\beta}i2\pi\left(1-Ve_{\beta-\alpha}\left(-z\right)\right)^{-1}e\left(z\left(x-s\right)\right)f\left(s\right)ds.
\end{eqnarray*}
Eq. (\ref{eq:gr}) follows from this.\end{proof}
\begin{rem}[Distribution Theory]
Fix $z\in\mathbb{C}\backslash\mathbb{R}$, and $s\in\left(\alpha,\beta\right)$.
The Green's function $G\left(\cdot,s,z\right)\in L^{2}\left(\left[\alpha,\beta\right],H\right)$
is the unique solution to the differential equation
\begin{equation}
zf-\frac{1}{i2\pi}f'=\delta_{s}.\label{Sec-4-eq:grd}
\end{equation}
Here, $\delta_{s}$ is the Dirac measure supported at $s\in\left(\alpha,\beta\right)$. 

For $x\neq s$, $G$ is the homogeneous solution, and so
\begin{equation}
G\left(x,s,z\right)=\begin{cases}
c_{1}e\left(z\left(x-s\right)\right) & \alpha\leq x<s\leq\beta\\
c_{2}e\left(z\left(x-s\right)\right) & \alpha\leq s<x\leq\beta
\end{cases}\label{eq:tmp1}
\end{equation}
where $c_{1}$ and $c_{2}$ are independent of $x$. 

Moreover, from the theory of distributions, (\ref{Sec-4-eq:grd})
implies that 
\begin{equation}
G\left(s-,s,z\right)-G\left(s+,s,z\right)=i2\pi;\label{Sec-4-eq:jmp}
\end{equation}
and by the boundary condition (\ref{Sec-4-eq:bd1}), 
\begin{equation}
G\left(\beta,s,z\right)=VG\left(\alpha,s,z\right).\label{eq:bd}
\end{equation}
Combining (\ref{eq:tmp1}), (\ref{Sec-4-eq:jmp}) and (\ref{eq:bd}),
we get 
\begin{eqnarray}
c_{1} & = & i2\pi\left(1-Ve_{\beta-\alpha}\left(-z\right)\right)^{-1}\label{eq:c1}\\
c_{2} & = & i2\pi\left(\left(1-Ve_{\beta-\alpha}\left(-z\right)\right)^{-1}-1\right)\label{eq:c2}
\end{eqnarray}
which in turn yields (\ref{eq:gr}).
\end{rem}
From the Green's function, one may reconstruct the projection-valued
measure associated with $P_{V}$. Our formula for this measure involves
the spectral resolution $E_{V}$ in $H$ of the boundary unitary,
written in the form 
\begin{equation}
V=\int_{[0,1)}e(\lambda)E_{V}(d\lambda)\label{Sec-4-eq:spV}
\end{equation}
and the projections 
\begin{equation}
E_{\tfrac{\lambda+m}{\beta-\alpha}}f:=\left(\beta-\alpha\right)^{-1}\left(\int_{\alpha}^{\beta}f(t)e\left(\tfrac{\lambda+m}{\beta-\alpha}t\right)dt\right)e_{\tfrac{\lambda+m}{\beta-\alpha}}\label{Sec-4-eq:E_T}
\end{equation}
in $L^{2}\left(\left[\alpha,\beta\right]\right),$ $m\in\mathbb{Z},$
onto the subspaces spanned by the unit vectors 
\[
f_{\lambda+m}(x):=\left(\beta-\alpha\right)^{-1/2}e\left(\tfrac{\lambda+m}{\beta-\alpha}x\right).
\]
We use $[0,1)$ in the spectral resolution of $V$, to indicate that,
if $1$ is an eigenvalue of $V,$ then the corresponding atom of $E_{V}$
is located at $0$ and not at $1.$ For any $\lambda\in\mathbb{R},$
the functions $f_{\lambda+m},m\in\mathbb{Z}$ form an orthonormal
basis for $L^{2}\left(\left[\alpha,\beta\right]\right)$ and $T_{a}f_{n}=e\left(\tfrac{na}{\beta-\alpha}\right)f_{n}$
for all $a\in\mathbb{R}$ and all $n\in\mathbb{Z}.$ 
\begin{thm}
\label{Sec-4-thm:Spectral-thm-for-P} Let $P_{V}$ be the selfadjoint
extension of $P_{0}$, in $L^{2}\left(\left[\alpha,\beta\right]\right)\otimes H,$
associated with the boundary unitary operator $V:H\rightarrow H$
by (\ref{eq:P_U}) and (\ref{eq:dom-P_U}). Suppose 
\begin{eqnarray*}
P_{V} & = & \int_{\mathbb{R}}\lambda F\left(d\lambda\right)
\end{eqnarray*}
where $F\left(d\lambda\right)$ is the projection-valued measure of
$P_{V}$. Then, for all $-\infty<\mu<\nu<+\infty$, 
\begin{equation}
\tfrac{1}{2}\left(F\left(\left(\mu,\nu\right)\right)+F\left(\left[\mu,\nu\right]\right)\right)=\int_{\left[0,1\right)}\sum_{m\in\mathbb{Z}}\chi_{\left(\mu,\nu\right)}\left(\tfrac{\lambda+m}{\beta-\alpha}\right)E_{\tfrac{\lambda+m}{\beta-\alpha}}\otimes E_{V}\left(d\lambda\right),\label{Sec-4-eq:E_P}
\end{equation}
where $E_{V}\left(d\lambda\right)$ is as in (\ref{Sec-4-eq:spV})
and $E_{\tfrac{\lambda+m}{\beta-\alpha}}$ is as in (\ref{Sec-4-eq:E_T}). \end{thm}
\begin{proof}
Let $-\infty<\mu<\nu<\infty.$ By Stone's formula \cite[pages 237 and 264]{ReSi81}
we have 
\begin{equation}
\tfrac{1}{2}\left(F\left(\left(\mu,\nu\right)\right)+F\left(\left[\mu,\nu\right]\right)\right)=\lim_{b\searrow0}\int_{\mathbb{R}}\chi_{(\mu,\nu)}(a)\left(\left(\overline{z}-P_{V}\right)^{-1}-\left(z-P_{V}\right)^{-1}\right)da,\label{Sec-4-eq:Stone-formula}
\end{equation}
where $z=z(a,b):=a+ib.$ It follows from (\ref{eq:rf}) and (\ref{Sec-4-eq:g-of-alpha})
that for $w\in\mathbb{C}\setminus\mathbb{R}$ and $f\in L^{2}\left(\left[\alpha,\beta\right],H\right)$
we have 
\begin{eqnarray*}
\left(w-P_{V}\right)^{-1}f(x) & = & i2\pi M_{w}\int_{\alpha}^{\beta}e\left(w\left(x-s\right)\right)f(s)ds\\
 &  & -i2\pi\int_{\alpha}^{x}e\left(w\left(x-s\right)\right)f(s)ds
\end{eqnarray*}
where $x\in[\alpha,\beta],$ 
\[
M_{w}:=\left(1-e\left(-Lw\right)V\right)^{-1},
\]
and $L:=\beta-\alpha.$ Hence, 
\[
\left(\left(\overline{z}-P_{V}\right)^{-1}-\left(z-P_{V}\right)^{-1}\right)f(x)=A+B
\]
where 
\begin{align*}
A & :=i2\pi M_{\overline{z}}\int_{\alpha}^{\beta}e\left(\overline{z}\left(x-s\right)\right)f(s)ds-i2\pi M_{z}\int_{\alpha}^{\beta}e\left(z\left(x-s\right)\right)f(s)ds\\
\end{align*}
and 
\[
B:=-i2\pi\int_{\alpha}^{x}\left(e\left(\overline{z}\left(x-s\right)\right)-e\left(z\left(x-s\right)\right)\right)f(s)ds.
\]
Now $\left|e\left(\overline{z}\left(x-s\right)\right)-e\left(z\left(x-s\right)\right)\right|=2\left|\sin\left(b\left(x-s\right)\right)\right|\leq2bL,$
so $B\to0$ as $b\to0$ uniformly in $a.$ Consequently, when calculating
\[
\tfrac{1}{2}\left(F[\mu,\nu]+F(\mu,\nu)\right)f(x)=\lim_{b\searrow0}\int_{\mathbb{R}}\chi_{(\mu,\nu)}(a)Ada.
\]
We may consider 
\[
i2\pi\left(M_{\overline{z}}-M_{z}\right)\int_{\alpha}^{\beta}e\left(z\left(x-s\right)\right)f(s)ds
\]
in place of $A$ because the norm of 
\[
\int_{\mu}^{\nu}i2\pi M_{z}\left(\int_{\alpha}^{\beta}e\left(z\left(x-s\right)\right)f(s)ds-\int_{\alpha}^{\beta}e\left(\overline{z}\left(x-s\right)\right)f(s)ds\right)da
\]
is bounded above by 
\[
\int_{\mathbb{R}}\chi_{(\mu,\nu)}(a)\left\Vert M_{z}\right\Vert 2bLda=\int_{\mathbb{R}}\chi_{(\mu,\nu)}(a)\left\Vert \left(1-e\left(-Lz\right)V\right)^{-1}\right\Vert 2bLda\underset{b\searrow0}{\longrightarrow}0.
\]
The limit $=0$ because 
\begin{eqnarray*}
\frac{b}{1-e\left(\lambda-L\left(a+ib\right)\right)} & = & \frac{b}{1-e^{2\pi b}e\left(\lambda-La\right)}\\
 & = & -\frac{be^{-2\pi b}e\left(La-\lambda\right)}{1-e^{-2\pi b}e\left(La-\lambda\right)}\\
 & = & -b\sum_{m=1}^{\infty}e^{-2\pi mb}e_{m}\left(La-\lambda\right)
\end{eqnarray*}
 and 
\begin{align*}
 & \left|\int_{\mu}^{\nu}-b\sum_{m=1}^{\infty}e^{-2\pi mb}e^{i2\pi(la-\lambda)m}da\right|\\
 & =\left|-b\sum_{m=1}^{\infty}e^{-2\pi mb}e^{-i2\pi\lambda m}\tfrac{1}{i2\pi Lm}\left(e^{i2\pi mL\nu}-e^{i2\pi mL\nu}\right)\right|\\
 & \leq\tfrac{b}{\pi L}\sum_{m=1}^{\infty}e^{-2\pi mb}\tfrac{1}{m}\\
 & =\tfrac{b}{\pi L}\log\left(1-e^{-2\pi b}\right)
\end{align*}
which $\to0$ (uniformly in ) as $b\searrow0.$ Hence, 
\begin{eqnarray}
 &  & \tfrac{1}{2}\left(F\left(\left(\mu,\nu\right)\right)+F\left(\left[\mu,\nu\right]\right)\right)f(x)\label{Sec-4-eq:intermediate}\\
 &  & =\lim_{b\searrow0}\int_{\mathbb{R}}i2\pi\chi_{(\mu,\nu)}(a)\left(M_{\overline{z}}-M_{z}\right)\int_{\alpha}^{\beta}e\left(z\left(x-s\right)\right)f(s)dsda\nonumber 
\end{eqnarray}
Using the spectral resolution (\ref{Sec-4-eq:spV}) of $V$ we have 

\begin{align*}
M_{\overline{z}}-M_{z} & =\left(1-e\left(-L\overline{z}\right)V\right)^{-1}-\left(1-e\left(-Lz\right)V\right)^{-1}\\
 & =\int_{[0,1)}\frac{1}{1-e(\lambda-La)e(iLb)}-\frac{1}{1-e(\lambda-La)e(-iLb)}E_{V}(d\lambda)\\
 & =\int_{[0,1)}Q\left(r_{b},\lambda-La\right)E_{V}(d\lambda),
\end{align*}
where $r_{b}:=e(-iLb)=e^{2\pi Lb}$ and 
\[
Q(r,\theta):=\frac{1-r^{2}}{1-2r\cos\left(2\pi\theta\right)+r^{2}}
\]
is the Poisson kernel for the unit circle. Consequently, 
\begin{align*}
 & \tfrac{1}{2}\left(F\left(\left(\mu,\nu\right)\right)+F\left(\left[\mu,\nu\right]\right)\right)f(x)\\
 & =\lim_{b\searrow0}\int_{\mathbb{R}}i2\pi\chi_{(\mu,\nu)}(a)\left(\int_{[0,1)}Q\left(r_{b},\lambda-La\right)E_{V}(d\lambda)\right)\int_{\alpha}^{\beta}e\left(z\left(x-s\right)\right)f(s)dsda\\
 & =\lim_{b\searrow0}\int_{\left[0,1\right)}\int_{0}^{\frac{1}{L}}\left(\sum_{m\in\mathbb{Z}}\chi_{\left(\mu,\nu\right)}\left(\tfrac{m}{L}+a\right)\right)Q\left(r_{b},\lambda-La\right)E_{V}\left(d\lambda\right)e(zx)\widehat{f}(z)da\\
 & =\int_{\left[0,1\right)}\left(\tfrac{1}{L}\sum_{m\in\mathbb{Z}}\chi_{\left(\mu,\nu\right)}\left(\tfrac{m+\lambda}{L}\right)\right)E_{V}\left(d\lambda\right)e\left(\tfrac{\lambda}{L}x\right)\widehat{f}\left(\tfrac{\lambda}{L}\right).
\end{align*}
Where $\widehat{f}(z):=\int_{\alpha}^{\beta}e\left(-zs\right)f(s)ds.$
The last equality follows from the Poisson  kernel $Q(r,\theta)$
being an approximate identity as $r\nearrow1$. 
\end{proof}

Theorem \ref{thm:1.3} is a special case of the following corollary
to Theorem \ref{Sec-4-thm:Spectral-thm-for-P}:
\begin{cor}
\label{Sec-4-cor:The-spectrum-of-P_V}The spectrum of $P_{V}$ equals
the set of $\lambda\in\mathbb{R}$ for which $e\left(\left(\beta-\alpha\right)\lambda\right)$
is in the spectrum of $V.$ \end{cor}
\begin{proof}
By (\ref{Sec-4-eq:spV}) the support of $E_{V}$ is the set $\left\{ \lambda\in[0,1]\mid e(\lambda)\in\mathrm{spectrum}(V)\right\} .$
Hence the result follows from (\ref{Sec-4-eq:E_P}).\end{proof}
\begin{example}
If the spectrum of $V$ equals the unit circle, then the spectrum
of $P_{V}$ equals the real line. In particular, the spectrum of $P_{V}$
need not be discrete. \end{example}
\begin{thm}
Let $P_{V}$ be the selfadjoint extension of $P_{0},$ in $\mathscr{H}:=L^{2}\left(\left[\alpha,\beta\right]\right)\otimes H,$
associated with the boundary unitary operator $V:H\rightarrow H$
by (\ref{eq:P_U}) and (\ref{eq:dom-P_U}). Then $P_{V}+1/\left(\beta-\alpha\right)$
is unitarily equivalent to $P_{V}$.\end{thm}
\begin{proof}
For all $f\in\mathscr{H}$, define 
\[
\hat{f}\left(\lambda+m,y\right):=\frac{1}{\beta-\alpha}\int_{\alpha}^{\beta}\overline{e\left(\tfrac{\lambda+m}{\beta-\alpha}t\right)}f(t,y)dt;
\]
See (\ref{Sec-4-eq:E_T}). By Theorem \ref{Sec-4-thm:Spectral-thm-for-P},
we have 
\[
f\left(x,y\right)=\int_{\left[0,1\right)}\sum_{m\in\mathbb{Z}}e_{\tfrac{\lambda+m}{\beta-\alpha}}\left(x\right)\otimes E_{V}\left(d\lambda\right)\hat{f}\left(\lambda+m,y\right).
\]
Moreover, for all $s\in\mathbb{R}$, 
\[
e\left(sP_{V}\right)f\left(x,y\right)=\int_{\left[0,1\right)}\sum_{m\in\mathbb{Z}}e\left(\tfrac{\lambda+m}{\beta-\alpha}s\right)e_{\tfrac{\lambda+m}{\beta-\alpha}}\left(x\right)\otimes E_{V}\left(d\lambda\right)\hat{f}\left(\lambda+m,y\right).
\]
Let $U$ be the unitary determined by 
\[
Uf\left(x,y\right)=\int_{\left[0,1\right)}\sum_{m\in\mathbb{Z}}e_{\tfrac{\lambda+m+1}{\beta-\alpha}}\left(x\right)\otimes E_{V}\left(d\lambda\right)\hat{f}\left(\lambda+m,y\right).
\]
A direct computation shows that
\[
e\left(sP_{V}\right)Uf=Ue\left(s\left(P_{V}+\frac{1}{\beta-\alpha}\right)\right)f.
\]
If, in addition, $f\in\mbox{dom}\left(P_{V}\right)$, then differentiating
the last equation at $s=0$ yields
\[
P_{V}Uf=U\left(P_{V}+\frac{1}{\beta-\alpha}\right)f.
\]
That is, $P_{V}U=U\left(P_{V}+1/\left(\beta-\alpha\right)\right)$.
This proves the theorem.
\end{proof}
For simplicity suppose $\beta-\alpha=1.$ The previous two results
suggests that $P_{V}$ is unitary equivalent to $\bigoplus_{k\in\mathbb{Z}}\left(L+k\right)$
for some bounded selfadjoint operator $0\leq L\leq1.$ Establishing
this as a consequence of Theorem \ref{Sec-4-thm:Spectral-thm-for-P}
provides an alternative proof of Theorem \ref{thm:1.2} and of Theorem
\ref{thm:1.3}. 
\begin{thm}
\label{Sec-4-thm:Sum-decomp-of-P}Suppose $\alpha=0$ and $\beta=1,$
then $\mathscr{H}=L^{2}\left[0,1\right]\otimes L^{2}\left[0,1\right].$
Let $P_{V}$ be the selfadjoint extension of $P_{0},$ in $\mathscr{H},$
associated with the boundary unitary operator $V:H\rightarrow H$
by (\ref{eq:P_U}) and (\ref{eq:dom-P_U}), 
\[
L_{k}:=P_{V}E_{P_{V}}\left(\left[k,k+1\right)\right),
\]
and $H_{k}:=E_{P_{V}}\left(\left[k,k+1\right)\right)\mathscr{H}.$
Clearly, $L_{k}$ is a selfadjoint operator acting in $H_{k},$ with
spectrum contained in $[k,k+1]$ and $k+1$ is not an eigenvalue of
$L_{k}.$ Furthermore, $P_{V}$ is unitary equivalent to $\bigoplus_{k\in\mathbb{Z}}\left(L_{0}+k\right).$ \end{thm}
\begin{proof}
Clearly, $P_{V}=\bigoplus_{k\in\mathbb{Z}}L_{k}$. For all $f\in\mathscr{H}$,
and $k\in\mathbb{Z}$, define
\[
\hat{f}\left(\lambda+k,y\right):=\int_{0}^{1}\overline{e_{\lambda+k}\left(t\right)}f\left(t,y\right)dt.
\]
By Theorem \ref{Sec-4-thm:Spectral-thm-for-P}, we have
\[
E_{P_{V}}\left(\left[k,k+1\right)\right)f\left(x,y\right)=\int_{\left[0,1\right)}e_{\lambda+k}\left(x\right)E_{V}\left(d\lambda\right)\hat{f}\left(\lambda+k,y\right);
\]
and 
\[
P_{V}E_{P_{V}}\left(\left[k,k+1\right)\right)f\left(x,y\right)=\int_{\left[0,1\right)}\left(\lambda+k\right)e_{\lambda+k}\left(x\right)E_{V}\left(d\lambda\right)\hat{f}\left(\lambda+k,y\right).
\]
Let $U_{k}:H_{k}\rightarrow H_{0}$ be the unitary operator determined
by 
\begin{eqnarray*}
U_{k}E_{P_{V}}\left(\left[k,k+1\right)\right)f\left(x,y\right) & = & U_{k}\left(\int_{\left[0,1\right)}e_{\lambda+k}\left(x\right)E_{V}\left(d\lambda\right)\hat{f}\left(\lambda+k,y\right)\right)\\
 & = & \int_{\left[0,1\right)}e_{\lambda}\left(x\right)E_{V}\left(d\lambda\right)\hat{f}\left(\lambda+k,y\right).
\end{eqnarray*}
A direct computation shows that
\[
\left(P_{V}+k\right)U_{k}E_{P_{V}}\left(\left[k,k+1\right)\right)f=U_{k}P_{V}E_{P_{V}}\left(\left[k,k+1\right)\right)f,
\]
i.e.,
\[
\left(P_{V}+k\right)U_{k}=U_{k}P_{V}E_{P_{V}}\left(\left[k,k+1\right)\right).
\]
we then get 
\[
\left(L_{0}+k\right)U_{k}=U_{k}L_{k}.
\]
Notice that $P_{V}U_{k}=P_{V}E_{P_{V}}\left(\left[0,1\right)\right)U_{k}$. 

Let $U:=\bigoplus_{k\in\mathbb{Z}}U_{k}$, and it follows that 
\[
UP_{V}U^{*}=U\left(\bigoplus_{k\in\mathbb{Z}}L_{k}\right)U^{*}=\bigoplus_{k\in\mathbb{Z}}U_{k}L_{k}U_{k}^{*}=\bigoplus_{k\in\mathbb{Z}}\left(L_{0}+k\right).
\]
This completes the proof. \end{proof}
\begin{rem}
The analogue of Theorem \ref{Sec-4-thm:Sum-decomp-of-P} for $P_{\theta}$
in $L^{2}\left([0,1]\right),$ $0\le\theta<1,$ determined by $P_{\theta}f=(i2\pi)^{-1}f'$
and $f(1)=e(\theta)f(0),$ states that $P_{\theta}$ is unitary equivalent
to $\bigoplus\left(L_{0}+k\right)$ in $\ell^{2}=\bigoplus H_{0},$
where $H_{k}=\mathbb{C}$ for all $k,$ and $L_{k}z=\left(\theta+k\right)z$
for $z\in\mathbb{C}=H_{k}.$ 
\end{rem}

\section{Spectral Pairs\label{Sec-5-sec:Spectral-Pairs}}

In this section we consider momentum operators on product domains
$[0,1]\times\Omega$ in $\mathbb{R}^{2},$ in the cases where $\Omega=\mathbb{R},$
$\Omega=[0,1],$ and where $\Omega$ is a certain fractal. We investigate
when the momentum operators in the $x$ and $y$ directions commute
in terms of the boundary unitaries. 

Recall, two (unbounded) selfadjoint operators $A$ and $B$ commute
if and only if their spectral measures commute. This is equivalent
to the commutation of the unitary one-parameter groups $e(aA)$ and
$e(bB)$ in the sense that 
\[
e(aA)e(bB)=e(bB)e(aA)
\]
for all $a,b$ in $\mathbb{R}.$ See, e.g., \cite{ReSi81}.

\subsection{The Infinite Strip }

In this section we consider the infinite strip $[0,1]\times\mathbb{R}.$
We obtain a complete classification of the commuting selfadjoint extensions
of $\left(i2\pi\right)^{-1}\partial_{x}$ and $\left(i2\pi\right)^{-1}\partial_{y}$
acting in $C_{c}^{\infty}\left([0,1]\times\mathbb{R}\right)$ in terms
of the boundary unitary associated with $\left(i2\pi\right)^{-1}\partial_{x}.$
Our method yields a complete list of the spectra of the infinite strip.
This set was shown to be a spectral set in \cite{Ped87}, but the
approach there only yields a partial list of the possible spectra.
 
\begin{thm}
\label{Sec-5-thm:Strip}Let $\mathscr{H}:=L^{2}\left(\left[\alpha,\beta\right],L^{2}\left(\mathbb{R}\right)\right)$.
Suppose $P=P_{U}:=\tfrac{1}{i2\pi}\frac{\partial}{\partial x}\Big|_{\mathrm{dom}\left(P_{U}\right)}$
is the selfadjoint extension corresponding to the unitary operator
$U:L^{2}\left(\mathbb{R}\right)\rightarrow L^{2}\left(\mathbb{R}\right)$.
Define $Q:=\tfrac{1}{i2\pi}\frac{\partial}{\partial y}\Big|_{\mathrm{dom}\left(Q\right)}$,
whose domain $\mathrm{dom}\left(Q\right)$ consists of all $f\in\mathscr{H}$,
such that $\frac{\partial}{\partial y}f$ (in the sense of distribution)
is in $\mathscr{H}$. Then $P$ and $Q$ commute if and only if $U$
is diagonalized via Fourier transform, as 
\begin{equation}
U=\int_{\mathbb{R}}e\left(\gamma\left(\lambda\right)\right)\left|e_{\lambda}\left\rangle \right\langle e_{\lambda}\right|d\lambda\label{Sec-5-eq:Urep}
\end{equation}
where $\gamma:\mathbb{R}\rightarrow[0,1)$ is a Borel function. \end{thm}
\begin{proof}
Note that $Q$ is selfadjoint, and the unitary one-parameter group
$e\left(tQ\right)$, $t\in\mathbb{R}$, is given by 
\[
e\left(tQ\right)f\left(x,y\right)=f\left(x,y+t\right)
\]
for all $f\in\mathscr{H}$. That is, 
\begin{equation}
e\left(tQ\right)=I\otimes\tau_{t}\label{eq:tR}
\end{equation}
where $\tau_{t}$ is the translation group in $L^{2}\left(\mathbb{R}\right)$.
Further, $Q$ is diagonalized via the Fourier transform 
\begin{equation}
Q=I\otimes\int_{-\infty}^{\infty}\lambda\,\left|e_{\lambda}\left\rangle \right\langle e_{\lambda}\right|d\lambda.\label{eq:Q}
\end{equation}
Here, $d\lambda$ denotes the Lebesgue measure on $\mathbb{R}$.

Now, suppose the two unitary one-parameter groups commute. By Lemma
\ref{sec-2-lem:expH}, 
\[
e\left(\left(\beta-\alpha\right)P\right)=I\otimes U.
\]
It follows that $I\otimes U$ commutes with $e\left(tQ\right)$, for
all $t\in\mathbb{R}$. From (\ref{eq:tR})-(\ref{eq:Q}), we see that
\begin{alignat*}{1}
\left(I\otimes U\right)e\left(tQ\right) & =\left(I\otimes U\right)\left(I\otimes\tau_{t}\right)=I\otimes U\tau_{t}\\
e\left(tQ\right)\left(I\otimes U\right) & =\left(I\otimes\tau_{t}\right)\left(I\otimes U\right)=I\otimes\tau_{t}U;
\end{alignat*}
hence $U\tau_{t}=\tau_{t}U$, for all $t\in\mathbb{R}$. Consequently,
e.g., \cite[Theorem 3.16]{SW71}, $U$ is diagonalized via the Fourier
transform, as in (\ref{Sec-5-eq:Urep}).

Conversely, suppose $U$ is given by (\ref{Sec-5-eq:Urep}). Fix $f\in\mathscr{H}$,
$t\in\mathbb{R}$. For all $n\in\mathbb{Z}$, we have 

\begin{equation}
\left\langle f,e\left(tQ\right)\left(1\otimes U^{n}\right)f\right\rangle =\int_{\left[0,1\right)}e_{n}\left(\lambda\right)\left\langle f,e\left(tQ\right)E_{U}\left(d\lambda\right)f\right\rangle ;\label{eq:m1}
\end{equation}
and 
\begin{equation}
\left\langle f,\left(1\otimes U^{n}\right)e\left(tQ\right)f\right\rangle =\int_{\left[0,1\right)}e_{n}\left(\lambda\right)\left\langle f,E_{U}\left(d\lambda\right)e\left(tQ\right)f\right\rangle .\label{Sec-5-eq:m2}
\end{equation}
Note that, by assumption, $U$ is diagonalized via Fourier transform,
and so $1\otimes U^{n}$ commutes with $e\left(tQ\right)$, for all
$n\in\mathbb{Z}$. Therefore, the two Borel measures, on the right-hand-side
of (\ref{eq:m1}) and (\ref{Sec-5-eq:m2}), have the same Fourier
coefficients; thus 
\begin{equation}
\left\langle f,e\left(tQ\right)E_{U}\left(d\lambda\right)f\right\rangle =\left\langle f,E_{U}\left(d\lambda\right)e\left(tQ\right)f\right\rangle .\label{eq:emeas}
\end{equation}
Multiplying $e\left(s\lambda\right)$ on both sides of (\ref{eq:emeas})
and integrating over $\left[0,1\right)$, we get 
\[
\int_{\left[0,1\right)}e\left(s\lambda\right)\left\langle f,e\left(tQ\right)E_{U}\left(d\lambda\right)f\right\rangle =\int_{\left[0,1\right)}e\left(s\lambda\right)\left\langle f,E_{U}\left(d\lambda\right)e\left(tQ\right)f\right\rangle 
\]
i.e., 
\[
\left\langle f,e\left(tQ\right)e\left(sP\right)f\right\rangle =\left\langle f,e\left(sP\right)e\left(tQ\right)f\right\rangle 
\]
for all $s\in\mathbb{R}$. 

Since $f$ and $t$ are arbitrary, we conclude that $e\left(sP\right)$
commutes with $e\left(tQ\right)$, for all $s,t\in\mathbb{R}$.
\end{proof}

\begin{rem}
\label{Sec-5-Remark5.2}To put $U$ in (\ref{Sec-5-eq:Urep}) into
the standard projection-valued measure form (\ref{Sec-4-eq:spV}),
let $E\left(d\lambda\right):=\left|e_{\lambda}\left\rangle \right\langle e_{\lambda}\right|d\lambda$,
and $E_{U}:=E\circ\gamma^{-1}$. Hence, 
\[
U=\int_{\mathbb{R}}e\left(\gamma\left(\lambda\right)\right)E\left(d\lambda\right)=\int_{[0,1)}e\left(\lambda\right)E\left(\gamma^{-1}\left(d\lambda\right)\right)=\int_{[0,1)}e\left(\lambda\right)E_{U}\left(d\lambda\right).
\]
Note that, for all $\varphi\in L^{2}\left(\mathbb{R}\right)$, and
all Borel set $\triangle$ in $\mathbb{R}$, 
\begin{eqnarray*}
\left\Vert E\left(\triangle\right)\varphi\right\Vert ^{2} & = & \int_{\triangle}\left|\widehat{\varphi}\left(\lambda\right)\right|^{2}d\lambda\\
\left\Vert E_{U}\left(\triangle\right)\varphi\right\Vert ^{2} & = & \int_{\gamma^{-1}\left(\triangle\right)}\left|\widehat{\varphi}\left(\lambda\right)\right|^{2}d\lambda.
\end{eqnarray*}

\begin{rem}
It follows from Theorem \ref{Sec-4-thm:Spectral-thm-for-P} and Remark
\ref{Sec-5-Remark5.2} that
\begin{align*}
P_{U} & =\int_{[0,1)}\sum_{m\in\mathbb{Z}}\left(\lambda+m\right)\left|e_{\lambda+m}\left\rangle \right\langle e_{\lambda+m}\right|\otimes E_{U}\left(d\lambda\right)\\
 & =\int_{\mathbb{R}}\sum_{m\in\mathbb{Z}}\left(\gamma(\lambda)+m\right)\left|e_{\gamma(\lambda)+m}\left\rangle \right\langle e_{\gamma(\lambda)+m}\right|\otimes\left|e_{\lambda}\left\rangle \right\langle e_{\lambda}\right|d\lambda
\end{align*}
Moreover, $Q$ is diagonalized via the Fourier transform, see (\ref{eq:Q}).
Therefore, the joint spectrum of $P_{U}$ and $Q$ is the closure
of the set 
\[
\Lambda_{\gamma}:=\left\{ \left(\begin{array}{c}
\gamma(\lambda)+m\\
\lambda
\end{array}\right)\:\Big|\: m\in\mathbb{Z},\lambda\in\mathbb{R}\right\} ,
\]
provided $\gamma$ has been chosen such that $e(\gamma(\lambda))$
is in the spectrum of $U$ for all $\lambda.$ 
\end{rem}
\end{rem}

\subsection{The Unit Square }

In this section we consider the unit square $[0,1]^{2}$. We obtain
a complete classification of the commuting extensions of $\tfrac{1}{i2\pi}\partial_{x}$
and $\tfrac{1}{i2\pi}\partial_{y}$ acting in $C_{c}^{\infty}\left([0,1]^{2}\right)$
in term of the boundary unitaries, see also \cite{JP00}. As a consequence
we recover the list of all possible spectra of $[0,1]^{2}$ first
obtained in \cite{JP99}. 
\begin{lem}
\label{lem:disI}Let $\left(X,\mathfrak{M}_{X},\mu\right)$ and $\left(Y,\mathfrak{M}_{Y},\nu\right)$
be measure spaces, where $\mu$ is a complex measure on $\mathfrak{M}_{X}$
and $\nu$ a positive measure on $\mathfrak{M}_{Y}$. Let $\pi:X\rightarrow Y$
be a measurable function.

Suppose there is a family of measures $\left\{ \psi\left(y,\cdot\right)\right\} _{y\in Y}$,
such that, 
\begin{enumerate}
\item \label{enu:c1}For all $y\in Y$, $\psi\left(y,\cdot\right)$ is supported
in $\pi^{-1}\left(y\right)$;
\item \label{enu:c2}For all $B\in\mathfrak{M}_{X}$, $\psi\left(\cdot,B\right)\in L^{1}\left(d\nu\right)$;
and
\begin{equation}
\mu\left(B\right)=\int\psi\left(y,B\right)\nu\left(dy\right).\label{Sec-5-eq:disI}
\end{equation}

\end{enumerate}

\begin{flushleft}
Then, for each $B\in\mathfrak{M}_{X}$, $\psi\left(\cdot,B\right)$
is uniquely determined. That is, if $\left\{ \psi'\left(y,\cdot\right)\right\} _{y\in Y}$
is another family of measures satisfying (\ref{enu:c1})-(\ref{enu:c2}),
then for all $B\in\mathfrak{M}_{X}$, 
\[
\psi\left(\cdot,B\right)=\psi'\left(\cdot,B\right),\quad\nu\mbox{ - }a.e.
\]

\par\end{flushleft}

\end{lem}
\begin{proof}
Fix $B\in\mathfrak{M}_{X}$. For all $F\in\mathfrak{M}_{Y}$, we have
\begin{equation}
\int_{F}\psi\left(y,B\right)\nu\left(dy\right)=\int\psi\left(y,B\cap\pi^{-1}\left(F\right)\right)\nu\left(dy\right)=\mu\left(B\cap\pi^{-1}\left(F\right)\right).\label{eq:mn1}
\end{equation}
Note the first equality follows from the assumption that $\psi\left(y,\cdot\right)$
is supported in $\pi^{-1}\left(y\right)$, for all $y\in Y$.

Consequently, $\mu\left(B\cap\pi^{-1}\left(\cdot\right)\right)\ll\nu$,
and the associated Radon-Nikodym derivative is $\psi\left(\cdot,B\right)$.
If $\left\{ \psi'\left(y,\cdot\right)\right\} _{y\in Y}$ is another
family of measures as stated, then (\ref{eq:mn1}) holds with $\psi'$
on the left-hand-side. The uniqueness of Radon-Nikodym derivative
then implies that $\psi\left(\cdot,B\right)=\psi'\left(\cdot,B\right)$,
$\nu$-a.e.\end{proof}
\begin{thm}
\label{Sec-5-thm:spectral-square}Let $\mathscr{H}:=L^{2}\left[0,1\right]\otimes L^{2}\left[0,1\right]$.
Let $P=P_{U}:=\tfrac{1}{i2\pi}\frac{\partial}{\partial x}\Big|_{\mathrm{dom}\left(P_{U}\right)}$
and $Q=Q_{V}:=\tfrac{1}{i2\pi}\frac{\partial}{\partial y}\Big|_{\mathrm{dom}\left(Q\right)}$
be the selfadjoint extensions corresponding to the boundary unitary
operators $U:L^{2}\left(I_{y}\right)\rightarrow L^{2}\left(I_{y}\right)$,
$V:L^{2}\left(I_{x}\right)\rightarrow L^{2}\left(I_{x}\right)$, respectively. 

Then $P$ and $Q$ commute if and only if there are $\alpha,\beta_{m}\in[0,1)$
such that 
\begin{equation}
Ve_{\alpha+m}=e(\beta_{m})e_{\alpha+m}\text{ and }U=e(\alpha)I\label{Sec-5-eq:1}
\end{equation}
or 
\begin{equation}
V=e(\alpha)I\text{ and }Ue_{\alpha+m}=e(\beta_{m})e_{\alpha+m}\label{Sec-5-eq:2}
\end{equation}
for all $m\in\mathbb{Z}.$ \end{thm}
\begin{proof}
Suppose 
\begin{eqnarray}
U & = & \int_{[0,1)}e\left(\lambda\right)E_{U}\left(d\lambda\right)\label{eq:spU1}\\
P & = & \int_{\mathbb{R}}\lambda E_{P}\left(d\lambda\right)\label{eq:spP1}
\end{eqnarray}
where $E_{U}$ and $E_{P}$ are the respective projection-valued measures.
By Theorem 4.3, for all Borel set $\triangle\subset\mathbb{R}$, 
\begin{eqnarray}
E_{P}\left(\triangle\right) & = & \int_{\left[0,1\right)}\Psi\left(\lambda,\triangle\right)\otimes E_{U}\left(d\lambda\right);\mbox{ where}\label{eq:Epk}\\
\Psi\left(\lambda,\triangle\right) & := & \sum_{m\in\mathbb{Z}}\chi_{\triangle}\left(\lambda+m\right)\left|e_{\lambda+m}\left\rangle \right\langle e_{\lambda+m}\right|.\label{Sec-5-eq:psik}
\end{eqnarray}
Let $f\otimes g\in\mathscr{H}$, then 
\begin{eqnarray}
\left\langle f\otimes g,\left(V\otimes I\right)E_{P}\left(\triangle\right)f\otimes g\right\rangle  & = & \int_{\left[0,1\right)}\left\langle f,V\Psi\left(\lambda,\triangle\right)f\right\rangle \left\Vert E_{U}\left(d\lambda\right)g\right\Vert ^{2}\label{eq:cm1-1}\\
\left\langle f\otimes g,E_{P}\left(\triangle\right)\left(V\otimes I\right)f\otimes g\right\rangle  & = & \int_{\left[0,1\right)}\left\langle f,\Psi\left(\lambda,\triangle\right)Vf\right\rangle \left\Vert E_{U}\left(d\lambda\right)g\right\Vert ^{2}.\label{eq:cm2-1}
\end{eqnarray}

Now, suppose $e\left(sP\right)$ and $e\left(tQ\right)$ commute,
for all $s,t\in\mathbb{R}$. In particular, by Lemma \ref{sec-2-lem:expH}
\[
V\otimes I=e\left(Q\right)
\]
so that $V\otimes I$ commutes with $e\left(sP\right)$, $s\in\mathbb{R}$.
Similarly, $I\otimes U$ commutes with $e\left(tQ\right)$, $t\in\mathbb{R}$.

Hence, the two complex Borel measures on the left-hand-side of (\ref{eq:cm1-1})-(\ref{eq:cm2-1})
are identical. We denote this measure by $\mu$. Also, let 
\[
\nu\left(d\lambda\right):=\left\Vert E_{U}\left(d\lambda\right)g\right\Vert ^{2}.
\]
Define $\pi:\left(\mathbb{R},\mu\right)\rightarrow\left(\mathbb{T}\cong\left[0,1\right),\nu\right)$
as the quotient map. Set 
\begin{eqnarray*}
\psi_{1}\left(\lambda,\cdot\right) & := & \left\langle f,V\Psi\left(\lambda,\cdot\right)f\right\rangle \\
\psi_{2}\left(\lambda,\cdot\right) & := & \left\langle f,\Psi\left(\lambda,\cdot\right)Vf\right\rangle .
\end{eqnarray*}
Then, for $j=1,2$, we have
\begin{enumerate}
\item For all $\lambda\in\left[0,1\right)$, $\psi_{j}\left(\lambda,\cdot\right)$
is supported in $\pi^{-1}\left(\lambda\right)=\lambda+\mathbb{Z}$;
see (\ref{Sec-5-eq:psik});
\item $\psi_{j}\left(\cdot,\triangle\right)\in L^{\infty}\left(\nu\right)$,
and 
\[
\mu\left(\triangle\right)=\int_{\left[0,1\right)}\psi_{j}\left(\lambda,\triangle\right)\nu\left(d\lambda\right);
\]
see (\ref{eq:cm1-1})-(\ref{eq:cm2-1}).
\end{enumerate}
Thus, by Lemma \ref{lem:disI}, $\left\langle f,V\Psi\left(\lambda,\triangle\right)f\right\rangle =\left\langle f,\Psi\left(\lambda,\triangle\right)Vf\right\rangle $,
$\nu$-a.e. Since $f$ is arbitrary, we conclude that 
\begin{equation}
V\Psi\left(\lambda,\triangle\right)=\Psi\left(\lambda,\triangle\right)V,\quad\nu\mbox{ - }a.e.\label{Sec-5eq:aeV}
\end{equation}
for each Borel set $\triangle\subset\mathbb{R}$.

Note that $\Psi\left(\lambda,\cdot\right)$ is a resolution of identity
in $L^{2}\left[0,1\right]$, thus (\ref{Sec-5eq:aeV}) implies that
there exists $\beta_{\lambda+m}\in\left[0,1\right)$, $m\in\mathbb{Z}$,
such that 
\begin{equation}
Ve_{\lambda+m}=e\left(\beta_{\lambda+m}\right)e_{\lambda+m},\quad\nu\mbox{ - }a.e.\label{eq:digV}
\end{equation}
Let $S$ be the set of $\lambda\in\left[0,1\right)$ such that (\ref{Sec-5eq:aeV})
holds, thus, $\nu\left(S^{c}\right)=0$. We proceed to show there
are two possibilities:

\textbf{Case 1.} $S=\left\{ \alpha\right\} $, i.e., a singleton.
Then (\ref{eq:digV}) yields 
\[
V=\sum_{m\in\mathbb{Z}}e\left(\beta_{m}\right)\left|e_{\alpha+m}\left\rangle \right\langle e_{\alpha+m}\right|.
\]
Moreover, since $\nu_{g}\left(d\lambda\right)=\left\Vert E_{U}\left(d\lambda\right)g\right\Vert ^{2}$
is supported at $\left\{ \alpha\right\} $ and $g$ was arbitrary,
it follows that 
\[
U=e\left(\alpha\right)I.
\]
This yields (\ref{Sec-5-eq:1}).

\textbf{Case 2.} $S$ consists of more than one point. Let $\lambda,\lambda'$
be distinct points in $S$. Then 
\begin{eqnarray*}
 &  & \left(e\left(\beta_{\lambda+m}\right)-e\left(\beta_{\lambda'+m'}\right)\right)\left\langle e_{\lambda'+m'},e_{\lambda+m}\right\rangle \\
 & = & \left\langle e_{\lambda'+m'},Ve_{\lambda+m}\right\rangle -\left\langle V^{*}e_{\lambda'+m'},e_{\lambda+m}\right\rangle \\
 & = & 0.
\end{eqnarray*}
Since $\left\langle e_{\lambda'+m'},e_{\lambda+m}\right\rangle \neq0$,
there is a constant $\alpha\in\left[0,1\right)$, such that $\beta_{\lambda+m}=\alpha$,
for all $\lambda\in\left[0,1\right)$, and $m\in\mathbb{Z}$. That
is,
\[
V=e\left(\alpha\right)I.
\]

We then run through the argument used in the proof, starting with
the fact that $I\otimes U$ commutes with $e\left(tQ\right)$, $t\in\mathbb{R}$.
It follows that
\[
U=\sum_{m\in\mathbb{Z}}e\left(\beta_{n}\right)\left|e_{\alpha+n}\left\rangle \right\langle e_{\alpha+n}\right|.
\]
This yields (\ref{Sec-5-eq:2}).

The converse is essentially trivial. For example, if (\ref{Sec-5-eq:1}),
then the functions $e_{\alpha+m}\otimes e_{\beta_{m}+n},$ $m,n\in\mathbb{Z}$
is a complete set of joint eigenfunctions for $P_{U}$ and $Q_{V}.$ \end{proof}
\begin{rem}
In case (\ref{Sec-5-eq:1}) the joint spectrum of $P$ and $Q$ is
\begin{equation}
\left\{ \left(\begin{array}{c}
\alpha+m\\
\beta_{m}+n
\end{array}\right)\mid m,n\in\mathbb{Z}\right\} ,\label{eq:sp1}
\end{equation}
since Theorem \ref{Sec-4-thm:Spectral-thm-for-P} in this case states
that 
\begin{eqnarray*}
P & = & \sum_{m\in\mathbb{Z}}\left(m+\alpha\right)\left|e_{\alpha+m}\left\rangle \right\langle e_{\alpha+m}\right|\otimes I\\
 & = & \sum_{m,n\in\mathbb{Z}}\left(m+\alpha\right)\left|e_{\alpha+m}\left\rangle \right\langle e_{\alpha+m}\right|\otimes\Bigl(\left|e_{\beta_{m}+n}\left\rangle \right\langle e_{\beta_{m}+n}\right|\Bigr)\\
Q & = & \sum_{m,n\in\mathbb{Z}}\left(\beta_{m}+n\right)\Bigl(\left|e_{\alpha+m}\left\rangle \right\langle e_{\alpha+m}\right|\Bigr)\otimes\Bigl(\left|e_{\beta_{m}+n}\left\rangle \right\langle e_{\beta_{m}+n}\right|\Bigr).
\end{eqnarray*}
Similarly, in case (\ref{Sec-5-eq:2}) the joint spectrum is 
\begin{equation}
\left\{ \left(\begin{array}{c}
\beta_{n}+m\\
\alpha+n
\end{array}\right)\mid m,n\in\mathbb{Z}\right\} .\label{eq:sp2}
\end{equation}
That this is the possible joint spectra was established in \cite{JP00}
by a different method. 
\begin{rem}
\label{Sec-5:Remark-Square-geometric}Suppose (\ref{Sec-5-eq:2}),
then $U$ is unitary equivalent to $\widetilde{U}e_{m}=e\left(\beta_{m}\right)e_{m}.$
Hence $P_{U}$ is unitary equivalent to $P_{\widetilde{U}}$ by Theorem
\ref{Sec-2-thm:Unitary-equivalence}. Furthermore, $\widetilde{U}$
is a geometric boundary condition, more precisely, a rotation if and
only if there is a real number $r,$ such that $\beta_{m}$ is the
fractional part $\left\langle rm\right\rangle $ of $rm$ for all
$m.$ See Remark \ref{Sec-3-remark:Geometric-Boundary-Condition}
and Remark \ref{Sec-3-Example:rotations}.
\end{rem}
\end{rem}

\subsection{A Fractal}

Let $\mu$ be a probability measure with support $C\subset\mathbb{R}.$
Suppose the functions 
\[
e_{\lambda},\lambda\in\Lambda
\]
form an orthonormal basis for $L^{2}(\mu).$ Let $Q$ be the selfadjoint
operator determined by 
\[
Q\left(\sum_{\lambda}c_{\lambda}g_{\lambda}\otimes e_{\lambda}\right)=\sum_{\lambda}\lambda c_{\lambda}g_{\lambda}\otimes e_{\lambda}
\]
whose domain is the set of all $g\in L^{2}\left(\left[0,1\right]\right)$
and all finite sums $\sum_{\lambda}c_{\lambda}g_{\lambda}\otimes e_{\lambda}$
with $c_{\lambda}\in\mathbb{C},$ $g_{\lambda}\in L^{2}\left(\left[0,1\right]\right),$
and $\lambda\in\Lambda.$ Then $Q$ is essentially selfadjoint and
$Qf=\tfrac{1}{i2\pi}\partial_{y}f$ for any $f=g\otimes\sum_{\lambda}c_{\lambda}e_{\lambda}.$
See also Appendix \ref{sec-B:Questions}. We also denote the closure
of this operator by $Q.$ 
\begin{thm}
\label{Sec-5-thm:Fractal}Let $U$ be a unitary on $H=L^{2}(\mu)$
and let $P_{U}$ be the corresponding selfadjoint extension of $P_{0}$,
in $L^{2}\left(\left[0,1\right]\right)\otimes H$ determined by (\ref{eq:P_U})
and (\ref{eq:dom-P_U}). Then $P_{U}$ and $Q$ commute if and only
if the function $e_{\lambda},$ $\lambda\in\Lambda$ are eigenfunctions
for $U.$ \end{thm}
\begin{proof}
Since $e(tQ)f\otimes e_{\lambda}=e(t\lambda)f\otimes e_{\lambda}$
for all $f\in L^{2}\left(\left[0,1\right]\right)$ and all $\lambda\in\Lambda$
it follows that 
\begin{equation}
e(tQ)=I\otimes e(t\widetilde{Q})\label{Sec-5-eq:random}
\end{equation}
where $\widetilde{Q}$ acting in $L^{2}(\mu)$ is determined by $\widetilde{Q}e_{\lambda}=\lambda e_{\lambda}$
for $\lambda\in\Lambda.$ 

By Lemma \ref{sec-2-lem:expH} $e(P)=I\otimes U,$ so it follows from
$e(tQ)e(P)=e(P)e(tQ)$ and (\ref{Sec-5-eq:random}) that 
\[
e(t\widetilde{Q})Ue_{\lambda}=e(t\lambda)Ue_{\lambda}.
\]
Consequently, $Ue_{\lambda}$ is a multiple of $e_{\lambda}$. 

The converse is trivial, see the proof of Theorem \ref{Sec-5-thm:spectral-square}.\end{proof}
\begin{rem}
\label{Sec-5:remark-spectrum} If $\gamma:\Lambda\to[0,1)$ is such
that $Ue_{\lambda}=e(\gamma(\lambda))e_{\lambda},$ then the functions
\[
e_{\gamma(\lambda)+m}\otimes e_{\lambda},m\in\mathbb{Z},\lambda\in\Lambda
\]
form an orthonormal basis for $L^{2}\left(\left[0,1\right]\right)\otimes L^{2}\left(\mu\right)$
consisting of joint eigenfunctions for $P_{U}$ and $Q$. Consequently,
the joint spectrum of $P_{U}$ and $Q$ is the closure of the set
of (joint) eigenvalues 
\begin{equation}
\Lambda_{\gamma}:=\left\{ \left(\begin{array}{c}
\gamma(\lambda)+m\\
\lambda
\end{array}\right)\mid m\in\mathbb{Z},\lambda\in\Lambda\right\} .\label{Sec-5-eq:Cantor-spectrum}
\end{equation}
\end{rem}
\begin{example}
Consider the Cantor set 
\[
C:=\left\{ \sum_{k=1}^{\infty}d_{k}4^{-k}\mid d_{k}\in\left\{ 0,3\right\} \right\} 
\]
and the set 
\[
\Lambda:=\left\{ \sum_{k=0}^{n}d_{k}4^{k}\mid d_{k}\in\left\{ 0,1\right\} \right\} .
\]
If $\mu$ is the measure determined by 
\[
\mu\left(\left[\sum_{k=1}^{n}d_{k}4^{-k},\sum_{k=1}^{n}d_{k}4^{-k}+4^{-n}\right]\right)=2^{-n}
\]
 for all $n\geq1$, where $d_{k}\in\{0,3\},$ then $C$ is the support
of $C$ and it was shown in \cite{JP98} that the functions $e_{\lambda},$
$\lambda\in\Lambda$ form an orthonormal basis for $L^{2}(\mu).$
The set $\Lambda$ is called a \emph{spectrum} of $\mu.$ If $\nu=m\otimes\mu,$
where $m$ is Lebesgue measure on the interval $[0,1]$, then it follows
from \cite{JP99} that 
\[
\left(\nu,\Lambda_{\gamma}\right)
\]
is a spectral pair, where $\Lambda_{\gamma}$ is determined by (\ref{Sec-5-eq:Cantor-spectrum}).
This, combined with Remark \ref{Sec-5:remark-spectrum}, gives an
explicit formula for the possible joint spectra of commuting pairs
$P_{U},Q$ in terms of the choice of a function $\gamma$ and the
spectrum $\Lambda$ of $\mu.$ Not all exponential basis for $L^{2}(\mu)$
are know, see the paper \cite{DHS09} and its references for constructions
of other exponential basis for $L^{2}(\mu).$ 
\end{example}
\appendix

\section{Selfadjoint Extensions\label{sec-A-:Selfadjoint-Extensions}}

Fix real numbers $\alpha<\beta$ and a Hilbert space $H.$ Consider
the Hilbert space 
\[
\mathscr{H}:=L^{2}\left([\alpha,\beta],H\right)=L^{2}\left(\left[\alpha,\beta\right]\right)\otimes H
\]
of $L^{2}-$functions $\left[\alpha,\beta\right]\to H$ equipped with
the inner product 
\[
\left\langle f\mid g\right\rangle :=\int_{\alpha}^{\beta}\left\langle f(x)\mid g(x)\right\rangle dx,
\]
where $\left\langle f(x)\mid g(x)\right\rangle $ is the inner product
in $H.$ We will consider selfadjoint restrictions of the operator
$P=P_{\max}$ determined by 
\begin{equation}
Pf:=\frac{1}{i2\pi}\frac{d}{dx}f=\frac{1}{i2\pi}f',\label{sec-A-eq:P-action}
\end{equation}
with the (maximal) domain 
\begin{equation}
\mathrm{dom}(P):=\left\{ f\in L^{2}\left([\alpha,\beta],H\right):f'\in L^{2}\left([\alpha,\beta],H\right)\right\} .\label{sec-A-eq:P-domain}
\end{equation}
Let $P_{\min}$ be the restriction of $P$ to the (minimal) domain
$\mathrm{dom}\left(P_{\min}\right):=C_{c}^{\infty}\left([\alpha,\beta]\right)\otimes H.$
Finally, let $P_{0}$ be the restriction of $P$ to the domain
\[
\mathrm{dom}\left(P_{0}\right):=\left\{ f\in\mathrm{dom}\left(P\right)\mid f(\alpha)=f(\beta)=0\right\} .
\]
Integrations by parts shows that $\left\langle P_{0}f\mid g\right\rangle =\left\langle f\mid P_{0}g\right\rangle $
for all $f,g$ in $\mathrm{dom}\left(P_{0}\right)$ and consequently
also $\left\langle P_{\min}f\mid g\right\rangle =\left\langle f\mid P_{\min}g\right\rangle $
for all $f,g$ in $\mathrm{dom}\left(P_{\min}\right).$ Hence, $P_{0}$
and $P_{\min}$ are densely defined symmetric operators in $L^{2}(\left[\alpha,\beta\right],H).$ 

Clearly, $P_{\min}$ is a restriction of $P_{0}.$ A consequence of
the next lemma is that $P_{\min}$ and $P_{0}$ have the same selfadjoint
extensions. 
\begin{lem}
\label{sec-A-lem:the-adjoint-of-P}We have 
\[
P_{\min}^{*}=P_{0}^{*}=P.
\]
Recall, $P=P_{\max}.$ \end{lem}
\begin{proof}
Fix $f\in L^{2}\left([\alpha,\beta],H\right)$ with $f'\in L^{2}\left([\alpha,\beta],H\right).$
For $g\in C_{c}^{\infty}([\alpha,\beta])\otimes H$ integration by
parts yields 
\begin{align*}
\left\langle P_{\min}g\mid f\right\rangle  & =\tfrac{1}{i2\pi}\int_{\alpha}^{\beta}\left\langle g'(x)\mid f(x)\right\rangle dx\\
 & =-\tfrac{1}{i2\pi}\int_{\alpha}^{\beta}\left\langle g(x)\mid f'(x)\right\rangle dx,
\end{align*}
since $g(\alpha)=g(\beta)=0.$ Consequently, $f$ is in $\mathrm{dom}\left(P_{\min}^{*}\right)$
and $P_{\min}^{*}f=\frac{1}{i2\pi}f'$. 

Conversely, fix $f\in D\left(P_{\min}^{*}\right).$ Let $g:=P_{\min}^{*}f$
and $G(x):=\int_{\alpha}^{x}g(t)dt.$ For $h\in C_{c}^{\infty}([\alpha,\beta)\otimes H$
we have $\left\langle P_{\min}h\mid f\right\rangle =\left\langle h\mid P_{\min}^{*}f\right\rangle =\left\langle h\mid g\right\rangle ,$
hence integration by parts leads to 
\begin{align*}
\int_{\alpha}^{\beta}\tfrac{1}{i2\pi}\left\langle h'(x)\mid f(x)\right\rangle dx & =\int_{\alpha}^{\beta}\left\langle h(x)\mid g(x)\right\rangle dx\\
 & =-\int_{\alpha}^{\beta}\left\langle h'(x)\mid G(x)\right\rangle dx
\end{align*}
since $h(\alpha)=h(\beta)=0.$ Consequently, 
\[
\int_{\alpha}^{\beta}\left\langle h'(x)\left|\,\frac{-1}{i2\pi}f(x)+G(x)\right.\right\rangle dx=0,
\]
for all $h\in C_{c}^{\infty}([\alpha,\beta])\otimes H.$ It follows
that $-\tfrac{1}{i2\pi}f(x)+G(x)$ is constant. Using the definition
of $G$ we conclude $f'$ exists and 
\[
\frac{1}{i2\pi}f'=G'=g=P_{0}^{*}f.
\]
Hence $f$ is in $\mathrm{dom}\left(P\right)$ and $P_{\min}^{*}f=Pf.$ 

Repeating this argument shows that $P_{0}^{*}=P.$ 
\end{proof}
When working with the von Neumann parametrization of the selfadjoint
extensions of a symmetric operator, it is important to start with
a closed operator, hence the following lemma is important. 
\begin{lem}
\label{Sec-A-lem:P-is-closed}The closure $\overline{P_{\min}}$ of
$P_{\min}$ equals $P_{0},$ in particular, $P_{0}$ is closed. \end{lem}
\begin{proof}
Using Lemma \ref{sec-A-lem:the-adjoint-of-P} we see that $\overline{P_{\min}}=P_{\min}^{**}=P^{*}=P_{0}^{**}=\overline{P_{0}}.$
Hence, it is sufficient to show that $P_{0}^{**}=P_{0}.$ Fix $f\in\mathrm{dom}\left(P_{0}^{**}\right).$
We must show $f\in\mathrm{dom}\left(P_{0}\right)$ and $P_{0}^{**}f=\tfrac{1}{i2\pi}f'.$
Let $g:=P_{0}^{**}f$ and $G(x):=\int_{\alpha}^{x}g(t)dt.$ For $h\in\mathrm{dom}\left(P_{0}^{*}\right)$
we have $\left\langle P_{0}^{*}h\mid f\right\rangle =\left\langle h\mid P_{0}^{**}f\right\rangle =\left\langle h\mid g\right\rangle .$
Hence integration by parts leads to
\begin{align*}
\int_{\alpha}^{\beta}\frac{1}{i2\pi}\left\langle h'(x)\left|\, f(x)\right.\right\rangle dx & =\int_{\alpha}^{\beta}\left\langle h(x)\left|\, g(x)\right.\right\rangle dx\\
 & =B(h,G)-\int_{\alpha}^{\beta}\left\langle h'(x)\left|\, G(x)\right.\right\rangle dx,
\end{align*}
where 
\[
B(h,G):=\left\langle h(\beta)\mid G(\beta)\right\rangle -\left\langle h(\alpha)\mid G(\alpha)\right\rangle .
\]
We can add a constant function $\phi(x)\equiv\psi$ to $h$ without
changing $h'.$ Hence, for such $\phi,$ 
\[
B(h,G)=B(h+\phi,G)=B(h,G)+B(\phi,G).
\]
 Consequently $B(\phi,G)=0$ for all constant functions $\phi(x)\equiv\psi.$
This means that 
\[
\left\langle \psi\mid G(\beta)-G)\alpha)\right\rangle =0
\]
for all $\psi\in H.$ So $G(\beta)=G(\alpha)=0.$ Now it follows,
as in the proof of the previous lemma, that $\tfrac{1}{i2\pi}f(x)-G(x)$
is constant in the $x$ variable, hence $f'$ exists and 
\[
\frac{1}{i2\pi}f'=G'=g=P_{0}^{**}f.
\]
In particular, $f'\in L^{2}\left([\alpha,\beta],H\right)$ and $P_{0}^{**}f=\tfrac{1}{i2\pi}f'.$

It remains to check that $f\left(\alpha\right)=f\left(\beta\right)=0.$
If $g\in\mathrm{dom}\left(P_{0}^{*}\right),$ then $\left\langle P_{0}^{*}g\mid f\right\rangle =\left\langle g\mid P_{0}^{**}f\right\rangle .$
This means that 
\[
\int_{\alpha}^{\beta}\left\langle g'(x)\left|\, f(x)\right.\right\rangle dx=-\int_{\alpha}^{\beta}\left\langle g(x)\left|\, f'(x)\right.\right\rangle dx.
\]
Integration by parts shows that the two sides of this equation differ
by
\[
B(f,g)=\left\langle g(\beta)\mid f(\beta)\right\rangle -\left\langle g(\alpha)\mid f(\alpha)\right\rangle .
\]
Hence, $B(f,g)=0$ for all $g\in\mathrm{dom}\left(P^{*}\right).$
Since, $\beta\neq\alpha,$ there are functions $g$ in $\mathrm{dom}\left(P^{*}\right)$
that are zero on one boundary point and an arbitrary element of $H$
on the other boundary point. Consequently, $f\left(\alpha\right)=f\left(\beta\right)=0.$ \end{proof}
\begin{lem}
\label{sec-A-lem:defect-spaces} The orthogonal complement of the
range of $P_{0}$ is the set of functions $f$ in $L^{2}(A)$ of the
form $f(x,y)=h(y)$ for some $h$ in $H.$ The orthogonal complement
of the range of $P_{0}\pm i$ is the set of all functions $f_{\pm}$
in $L^{2}\left([\alpha,\beta],H\right)$ such that $f_{\pm}(x)=\exp(\pm2\pi x)h,$
for $x$ in $[\alpha,\beta]$ and $h$ in $H.$ \end{lem}
\begin{proof}
Suppose $f$ is in the orthogonal complement to the range of $P_{0}.$
Then 
\[
\left\langle P_{0}g\mid f\right\rangle =0
\]
for all $g\in D(P_{0}),$ hence $f\in D\left(P_{0}^{*}\right)$ and
$P_{0}^{*}f=0.$ By Lemma \ref{sec-A-lem:the-adjoint-of-P} $f'=0.$
Solving this differential equation gives the desired conclusion. The
calculation of the orthogonal complement of the range of $P_{0}\pm i$
is similar. \end{proof}
\begin{prop}
\label{sec-A-prop:sa-extensions-vN} The selfadjoint extensions of
$P_{0}$ are parametrized by the unitaries $V:H\to H$. The selfadjoint
extension $P_{V}$ of $P_{0}$ corresponding to the unitary $V$ is
the restriction of $P=P_{\max}$ whose domain $\mathrm{dom}\left(P_{V}\right)$
consists of the functions
\[
f(x)+e^{2\pi\left(\beta-x\right)}h+e^{2\pi\left(x-\alpha\right)}Vh,
\]
where $f\in\mathrm{dom}\left(P_{0}\right)$ and $h\in H.$ The action
of $P_{V}$ is 
\[
P_{V}\left(f(x)+e^{2\pi\left(\beta-x\right)}h+e^{2\pi\left(x-\alpha\right)}Vh\right)=\frac{1}{i2\pi}f'(x)+ie^{2\pi\left(\beta-x\right)}h-ie^{2\pi\left(x-\alpha\right)}Vh
\]
where $f$ and $h$ are as above. \end{prop}
\begin{proof}
This is an application of the von Neumann index theory, see e.g.,
\cite{ReSi75} for an account of this theory. $P_{0}$ is densely
defined, since $C_{c}^{\infty}([\alpha,\beta])\otimes H\subseteq\mathrm{dom}(P_{0})$
and $P_{0}$ is closed by Lemma \ref{Sec-A-lem:P-is-closed}. The
deficiency spaces $\mathscr{D}_{\pm}\left(P_{0}\right):=\ker\left(P_{0}\mp iI\right)$
of $P_{0}$ are 
\begin{equation}
\mathscr{D}_{\pm}\left(P_{0}\right)=\left\{ f\in L^{2}\left([\alpha,\beta],H\right)\mid f(x)=\exp\left(\mp2\pi x\right)h\,,h\in H\right\} \label{sec-A-eq:defect-spaces}
\end{equation}
according to Lemma \ref{sec-A-lem:defect-spaces}. In particular,
$\dim\mathscr{D}_{+}\left(P_{0}\right)=\infty=\dim\mathscr{D}_{-}\left(P_{0}\right),$
and consequently, $P_{0}$ has selfadjoint extensions. 

By the von Neumann theory, the selfadjoint extensions of $P_{0}$
are parametrized by the partial isometries $W$ with initial space
$\mathscr{D}_{+}\left(P_{0}\right)$ and final space $\mathscr{D}_{-}\left(P_{0}\right).$
Specitically, the selfadjoint extension $P_{W}$ corresponding to
the partial isometry $W$ is the restriction of $P=P_{0}^{*}$ to
the domain 
\[
\mathrm{dom}\left(P_{W}\right):=\left\{ f+f_{+}+Wf_{+}\mid f\in\mathrm{dom}(P_{0}),f_{+}\in\mathscr{D}_{+}\left(P_{0}\right)\right\} .
\]
If $f_{+}(x):=e^{2\pi\left(\beta-x\right)}h$ and $f_{-}(x):=e^{2\pi\left(x-\alpha\right)}Vh$,
then
\[
4\pi\|f_{+}\|_{L^{2}\left([\alpha,\beta],H\right)}^{2}=\left(e^{4\pi(\beta-\alpha)}-1\right)\|h\|_{H}^{2}=4\pi\|Vh\|_{H}^{2}.
\]
The correspondance between $W$ and $V$ is determined by
\[
We^{2\pi\left(\beta-x\right)}h=e^{2\pi\left(x-\alpha\right)}Vh.
\]
The claims are now immediate. 
\end{proof}
Another way to describe the selfadjoint extensions of $P$ are through
boundary conditions.
\begin{prop}
\label{sec-A-prop:sa-extension-boundary} The selfadjoint extensions
of $P_{0}$ are parametrized by the unitaries $V:H\to H.$ The selfadjoint
extension $P_{V}$ of $P_{0}$ corresponding to the unitary $V$ is
the restriction of $P=P_{\max}$ with domain
\begin{equation}
\mathrm{dom}\left(P_{V}\right):=\left\{ f\in\mathrm{dom}\left(P\right)\mid f(\beta)=Vf(\alpha)\right\} .\label{Sec-A-eq:dom H_V}
\end{equation}
\end{prop}
\begin{proof}
\textcolor{black}{If $f\in\mathrm{dom}\left(P\right),$ then we saw
in the proof of Lemma \ref{sec-A-lem:the-adjoint-of-P} that 
\[
f(x)=h+\int_{\alpha(y)}^{x}Pf(t)dt,
\]
for some $h$ in $H.$ In particular, $f(\alpha)$ and $f(\beta)$
are well-defined elements of $H.$ }Integration by parts shows that
for $f,g\in\mathrm{dom}\left(P\right)$ we have 
\[
\left\langle Pf\mid g\right\rangle =B(f,g)+\left\langle f\mid Pg\right\rangle 
\]
where 
\[
B(f,g):=\frac{1}{i2\pi}\left\langle f(\beta)\mid g(\beta)\right\rangle -\left\langle f(\alpha)\mid g(\alpha)\right\rangle .
\]
\textcolor{black}{Since $h$ is arbitrary in $H$, the maps 
\[
f\in\mathrm{dom}(P)\to f(\alpha)\in H\text{ and }f\in\mathrm{dom}(P)\to f(\beta)\in H
\]
have dense ranges. Consequently,} the result follows from \cite[Theorem 7.1.13]{dO09}. 
\end{proof}
Let $V_{vN}$ be the unitary from Proposition \ref{sec-A-prop:sa-extensions-vN}
and let $V_{B}$ be the unitary from Proposition \ref{sec-A-prop:sa-extension-boundary}. 

The function 
\[
g(x)=f(x)+e^{2\pi\left(\beta-x\right)}h+e^{2\pi\left(x-\alpha\right)}V_{vN}h,
\]
from Proposition \ref{sec-A-prop:sa-extensions-vN} has boundary values
$f(\alpha)=e^{2\pi(\beta-\alpha)}h+V_{nN}h$ and $f(\beta)=h+e^{2\pi(\beta-\alpha)}V_{vN}h.$
Hence, $h=\left(e^{2\pi(\beta-\alpha)}+V_{vN}\right)^{-1}f(\alpha)$
and $f(\beta)=\left(1+e^{2\pi(\beta-\alpha)}V_{vN}\right)h.$ It follows
that
\[
V_{B}=\left(1+e^{2\pi(\beta-\alpha)}V_{vN}\right)\left(e^{2\pi(\beta-\alpha)}+V_{vN}\right)^{-1}.
\]
Conversely, 
\[
V_{vN}=\left(e^{2\pi(\beta-\alpha)}-V_{B}\right)^{-1}\left(e^{2\pi(\beta-\alpha)}V_{B}-1\right).
\]
Hence, Proposition \ref{sec-A-prop:sa-extensions-vN} and Proposition
\ref{sec-A-prop:sa-extension-boundary} are, in fact, equivalent.
In particular, we could have established Proposition \ref{sec-A-prop:sa-extension-boundary}
without appealing to \cite[Theorem 7.1.13]{dO09}.

\section{Questions\label{sec-B:Questions}}
\begin{problem}
A set $\Lambda$ is a tiling set for the square $\left[0,1\right]^{2}$
if $\bigcup_{\lambda\in\Lambda}\left(\lambda+\left[0,1\right]^{2}\right)=\mathbb{R}^{2}$
and the overlaps are null sets. It is known that $\Lambda$ is the
joint spectrum for some commuting extensions $P,Q$ as in Theorem
\ref{Sec-5-thm:spectral-square} if and only if $\Lambda$ is a tiling
set for the square, see \cite{JP99}, \cite{IP98}, \cite{LRW00}.
In case (\ref{Sec-5-eq:1}) with $\alpha=0$ and $\beta_{m}=\left\langle rm\right\rangle $
for some $r\in\mathbb{R}$ we see that both $U$ and $V$ are determined
by the geometric boundary conditions from Remark \ref{Sec-3-remark:Geometric-Boundary-Condition}.
I might be of interest to investigate the relationship between geometric
boundary conditions and the boundary unitary operators $U$ and $V$
in more detail. 
\begin{problem}
Suppose $(\mu,\nu)$ is a spectral pair. Then $f(x)=\int\widehat{f}(\lambda)e(\lambda x)d\nu(\lambda).$
Let 
\begin{equation}
e\left(bQ\right)f(x):=\int e\left(b\lambda\right)\widehat{f}(\lambda)e(\lambda x)d\nu(\lambda).\label{Appendix-eq:1}
\end{equation}
 Then $e\left(bQ\right)f(x)=f(x+b)$ for a.e. $x$ such that $x+b$
is in the support of $\mu.$ So, ignoring null sets, if $b_{n}\to0$
and $x+b_{n}$ is in the support of $\mu$ for all $n,$ then 
\begin{equation}
\frac{e\left(b_{n}Q\right)f(x)-f(x)}{b_{n}}\to f'(x)\label{Appendix-eq:2}
\end{equation}
but the limit also equals $i2\pi Qf(x).$ Hence, $Q$ determined by
$(\ref{Appendix-eq:1})$ can be thought of as a selfadjoint realization
of $\tfrac{1}{i2\pi}\tfrac{d}{dx}$ in $L^{2}(\mu).$ Theorem \ref{Sec-5-thm:Strip}
considers the case where $\mu$ is Lebesgue measure on the real line
and Theorem \ref{Sec-5-thm:Fractal} a certain restriction of $1/2-$dimensional
Hausdorff measure. A common generalization of these two cases is: \end{problem}
\begin{thm*}
Suppose $(\mu,\nu)$ is a spectral pair of measures on $\mathbb{R}$
and $Q$ is determined by $(\ref{Appendix-eq:1})$. Let $\mathscr{H}:=L^{2}\left(\left[0,1\right]\right)\otimes L^{2}(\mu),$
let $U$ be a boundary unitary in $L^{2}\left(\mu\right)$ and let
$P_{U}$ be the corresponding selfadjoint extension of $P_{0}.$ Then
$P_{U}$ and $I\otimes Q$ commute if and only if 
\[
Uf(x)=\int e(\gamma(\lambda))\widehat{f}(\lambda)e(\lambda x)d\nu(\lambda)
\]
 for some $\nu-$measurable function $\gamma:\mathbb{R}\to[0,1).$ 
\end{thm*}
Is there a way to generalize this to also include Theorem \ref{Sec-5-thm:spectral-square}
as a special case? 
\end{problem}

\section*{Acknowledgement}

The co-authors thank Palle Jorgensen for many useful conversations
related to this paper. \vfill\pagebreak

\bibliographystyle{amsalpha}
\bibliography{square}

\end{document}